\documentclass[12pt,reqno,oneside]{amsart} 
\usepackage[utf8]{inputenc} 

\usepackage{geometry} 
\usepackage{chngcntr}
\usepackage{amssymb}
\usepackage[foot]{amsaddr}

\usepackage{fancyhdr} 
\newtheorem{theorem}{Theorem}
\newtheorem{proposition}[theorem]{Proposition}
\newtheorem{lemma}[theorem]{Lemma}

\newtheorem*{mainthmi}{Theorem I}
\newtheorem*{mainthmiia}{Theorem II A}
\newtheorem*{mainthmiib}{Theorem II B}
\counterwithin{equation}{section}
\counterwithin{theorem}{section}
\counterwithin{equation}{section}
\allowdisplaybreaks

\title{Irreducible completely pointed modules of quantum groups of type $A$}

\author[V. Futorny]{Vyacheslav Futorny$_\MakeLowercase{a}$}
\author[J. Hartwig]{Jonas Hartwig$_\MakeLowercase{b}$}
\author[E. Wilson]{Evan Wilson$_\MakeLowercase{a}^*$}
\address{$^a$Department of Mathematics Univ. of S\~ao Paulo, Caixa Postal 66281, S\~ao Paulo, SP 05315-970 -- Brazil}
\address{$^b$Department of Mathematics, University of California, Riverside, CA 92521}
\address{$^*$Corresponding author. Ph. +001-919-627-5299}
\email{vfutorny@gmail.com,jonas.hartwig@gmail.com,wilsoneaster@gmail.com}
\date{\today} 

\begin{document}				
				\begin{abstract} We give a classification of all irreducible completely pointed $U_q(\mathfrak{sl}_{n+1})$ modules over a characteristic zero field in which $q$ is not a root of unity. 
				This generalizes the classification result of Benkart, Britten and Lemire in the non quantum case. We also show that any infinite-dimensional irreducible completely pointed $U_q(\mathfrak{sl}_{n+1})$ can be obtained from  some irreducible completely pointed module over the quantized Weyl algebra $A_{n+1}^q$.

\end{abstract}
\maketitle
\section*{Introduction}
Let $U_q(\mathfrak{g})$ be the quantum group of finite-dimensional semisimple Lie algebra over a characteristic zero field in which $q$ is not a root of unity. A $U_q(\mathfrak{g})$ weight module is called completely pointed if all of its weight spaces are one dimensional. This paper is a generalization of the classification given by Benkart, Britten and Lemire (\cite{BBL}) of infinite dimensional completely pointed modules of semisimple Lie algebras. In the Lie algebra case such modules can only exist if every ideal of $\mathfrak{g}$ is of type $A$ or $C$. In the current paper we consider the case of $U_q(\mathfrak{sl}_{n+1})$, i.e. the quantum group of type $A$.

Throughout the paper, we will make extensive use of certain generators $E_\alpha$ of $U_q(\mathfrak{g})$ introduced by Lusztig (\cite{L90}) where $\alpha$ ranges over the roots of $\mathfrak{g}$, which are analogues of the root vectors of $\mathfrak{g}$. These generators are not unique, but depend on a choice of reduced decomposition of the longest Weyl group element, $w_0$. For a fixed reduced decomposition of $w_0$, and a irreducible weight module $V$ we see that each $E_\alpha$ acts either locally nilpotently or injectively on $V$. This set is very much like root vectors parabolic subalgebra $\mathfrak{p}$ of $\mathfrak{sl}_{n+1}$ though $\mathfrak{sl}_{n+1}\not \subset U_q(\mathfrak{sl}_{n+1})$ as a Lie algebra so this correspondence is not precise. Nevertheless, we call the $U_q(\mathfrak{sl}_{n+1})$-subalgebra generated by locally nilpotent root vectors $U_q(\mathfrak{p})$ and there exists another $U_q(\mathfrak{sl}_{n+1})$-submodule $U_q(\mathfrak{u})$ where $\mathfrak{u}$ is analogous to the nilradical of $\mathfrak{p}$. The first main theorem of the paper is as follows, where $V^+$ is the $\mathfrak{u}$-invariant subset of $V$:
\begin{mainthmi}
Let $V$ be an irreducible, infinite-dimensional, completely-pointed $U_q(\mathfrak{sl}_{n+1})$-module and let $v^+\in V^+$ be given. Then the action of $U_q(\mathfrak{sl}_{n+1})$  on $V$ can be extended to a $U_q(\mathfrak{gl}_{n+1})$ action such that the following relations hold:
\begin{equation*}
E_{-\varepsilon_i+\varepsilon_j}E_{\varepsilon_i-\varepsilon_j}\cdot v^+=[\bar K_{i};1][\bar K_j;0]\cdot v^+.
\end{equation*}
Also, we have
\begin{equation*}
F_{i}E_{i}\cdot v_\lambda=[\bar K_{i};1][\bar K_{i+1};0]\cdot v_\lambda
\end{equation*}
for any weight vector $v_{\lambda}\in V$.
\end{mainthmi}
Using this theorem one sees, for example, that the action of the cyclic subalgebra $C(U_q(\mathfrak{sl}_{n+1}))$ is completely determined by the action of the $\bar K_i$ (see Lemma \ref{lem:kerpi_AnnV}), hence this gives a classification of irreducible completely pointed $U_q(\mathfrak{sl}_{n+1})$ modules (with all finite dimensional ones given in Proposition \ref{prop:highest}).

In our next two main results, we construct the infinite-dimensional completely pointed $U_q(\mathfrak{sl}_{n+1})$-modules. Let $A^q_{n+1}$ be the rank $n+1$ quantum Weyl algebra and $\pi$ be the homomorphism from $U_q(\mathfrak{gl}_{n+1})$ to $A_{n+1}^q$ (see \cite{Hayashi}) which restricts to $U_q(\mathfrak{sl}_{n+1})$. Then we have the following:
\begin{mainthmiia}
Let $W$ be an irreducible completely pointed $A_{n+1}^q$-module. Let $\pi^\ast W$ be the $U_q(\mathfrak{gl}_{n+1})$-module, given as the $\pi$-pullback of $W$. Then $\pi^\ast W$ is completely reducible, and each irreducible submodule is completely pointed, and occurs with multiplicity one.
\end{mainthmiia}
This gives a construction of irreducible infinite dimensional, completely pointed $U_q(\mathfrak{sl}_{n+1})$-modules. An application of Theorem I then gives the following, which completes our classification:
\begin{mainthmiib}
Any infinite-dimensional irreducible completely pointed $U_q(\mathfrak{sl}_{n+1})$ is isomorphic to a direct summand of $\pi^\ast W$ for some irreducible completely pointed $A_{n+1}^q$-module $W$.
\end{mainthmiib}

\section{Preliminaries}
Let $\mathbb{F}$ be a field of characteristic 0 closed under quadratic extensions and suppose $q\in\mathbb{F}$ is nonzero and not a root of unity. For us, $U_q(\mathfrak{gl}_{n+1})$ is the associative unital $\mathbb{F}$-algebra with generators $E_i,F_i,\bar K_j^{\pm 1}$, $i\in\{1,\ldots,n\}$, $j\in\{1,\ldots,n+1\}$ and defining relations
\begin{enumerate}
\item $\bar K_j E_i \bar K_j^{-1} = q^{\delta_{ij}-\delta_{j,i+1}}E_i,\bar K_j F_i \bar K_j^{-1} = q^{-(\delta_{ij}-\delta_{j,i+1})}F_i,i\in\{1,\ldots,n\},j\in\{1,\ldots,n+1\}$\\
\item $[E_i,F_j]=\delta_{ij}\frac{\bar K_i\bar K_{i+1}^{-1}-\bar K_i^{-1}\bar K_{i+1}}{q-q^{-1}},i,j\in\{1,\ldots,n+1\}$\\
\item $[E_i^\pm,E_j^\pm]=0 \quad \text{for $|i-j|>1$},$\\
\item $(E_i^\pm)^2E_j^\pm-[2]_qE_j^\pm E_i^\pm E_j^\pm + E_j^\pm (E_i^\pm)^2=0,\quad \text{for $|i-j|=1$,}$
\end{enumerate}
where $E_i^+:=E_i$, $E_i^-:=F_i$, and $[k]_q=\frac{q^k-q^{-k}}{q-q^{-1}},k\in \mathbb{Z}_{\geq 0}$. Then, $U_q(\mathfrak{sl}_{n+1})$ is the subalgebra of $U_q(\mathfrak{gl}_{n+1})$ generated by $E_i,F_i$, and $K_i:=\bar{K}_i\bar{K}_{i+1}^{-1}$.

For $\mathfrak{g}=\mathfrak{sl}_{n+1}$ recall the following automorphisms $T_i:U_q(\mathfrak{g})\to U_q(\mathfrak{g}), 1\leq i \leq n$ as given by Lusztig \cite{L90}:
\begin{equation*}
T_i(E_j)=\begin{cases}
		-F_iK_i, &\text{if }i=j\\
		q^{-1}E_jE_i-E_iE_j, &\text{if }|i-j|=1\\
		E_j, &\text{otherwise},
\end{cases}
\end{equation*}
\begin{equation*}
T_i(F_j)=\begin{cases}
		-K_i^{-1}E_i, &\text{if }i=j\\
		-F_jF_i+qF_iF_j, &\text{if }|i-j|=1\\
		F_j, &\text{otherwise},
\end{cases}
\end{equation*}

\begin{equation*}
T_i(K_j)=\begin{cases}
K_j^{-1}, \text{ if $i=j$,}\\
K_iK_j, \text{ if $|i-j|=1$,}\\
K_j, \text{ otherwise}.
\end{cases}
\end{equation*}
We also recall the braid relations satisfied by the $T_i$:
\begin{equation*}
T_iT_{i+1}T_i=T_{i+1}T_iT_{i+1},
\end{equation*}
\begin{equation*}
T_iT_j=T_jT_i, \text{if }|i-j|>1.
\end{equation*}

To each root $\alpha$ we assign a corresponding root vector $E_\alpha$ in using following method.  Let $w_0=s_{i_1}s_{i_2}\cdots s_{i_r}$ be a reduced decomposition of the longest Weyl group element.  Then every positive root occurs exactly once in the following sequence:
$$
\beta_1=\alpha_{i_1},\beta_2=s_{i_1}(\alpha_{i_2}),\dots, \beta_r=s_{i_1}s_{i_2}\cdots s_{i_{r-1}}(\alpha_{i_r}).
$$
The positive root vector $E_{\beta_k}$ is defined to be $T_{i_1}T_{i_2}\cdots T_{i_{k-1}}(E_{i_k})$, and the negative root vector $E_{-\beta_k}$ is defined by the same sequence of $T_i$'s acting on $F_{i_k}$. We choose $w_0=s_1s_2\cdots s_{n} s_1s_2\cdots s_{n-1}\cdots s_1s_2s_1$ as our reduced expression of the longest Weyl group element of $U_q(\mathfrak{sl}_{n+1})$, which gives the following sequence of positive roots:
\begin{gather*}
\varepsilon_1-\varepsilon_2, \varepsilon_1-\varepsilon_3, \varepsilon_1-\varepsilon_4,\dots,\varepsilon_1-\varepsilon_{n+1},\\
\varepsilon_2-\varepsilon_3, \varepsilon_2-\varepsilon_4,\dots, \varepsilon_2-\varepsilon_{n+1},\\
\cdots\\
\varepsilon_{n-1}-\varepsilon_{n}, \varepsilon_{n}-\varepsilon_{n+1},\\
\varepsilon_{n}-\varepsilon_{n+1}.
\end{gather*}

We recall also from \cite[Example 8.1.5]{CP}  the following identity:
$$
T_iT_{i+1}(E_i)=E_{i+1}.
$$
Using this fact, and the braid relations, one obtains the following simplified form for the root vectors:
\begin{gather*}
E_{\varepsilon_1-\varepsilon_2}=E_1,E_{\varepsilon_1-\varepsilon_3}=T_1(E_2),E_{\varepsilon_1-\varepsilon_4}=T_1T_2(E_3),\dots, E_{\varepsilon_1-\varepsilon_{n-1}}=T_1T_2\cdots T_{n-1}(E_{n}),\\
E_{\varepsilon_2-\varepsilon_3}=E_2,E_{\varepsilon_2-\varepsilon_4}=T_2(E_3),\dots, E_{\varepsilon_2-\varepsilon_{n+1}}=T_2T_3\cdots T_{n-1}(E_{n}),\\
\dots\\
E_{\varepsilon_{n-1}-\varepsilon_{n}}=E_{n-1},E_{\varepsilon_{n-1}-\varepsilon_{n+1}}=T_{n-1}(E_{n}),\\
E_{\varepsilon_{n}-\varepsilon_{n+1}}=E_{n}
\end{gather*}
and similarly for the negative root vectors. Let $1\leq i<j<k \leq n$. Double induction on $i$ and $j$ gives the following, where $[x,y]_v=xy-vyx$:
\begin{align}
E_{\varepsilon_i-\varepsilon_k}&=-[E_{\varepsilon_i-\varepsilon_j},E_{\varepsilon_j-\varepsilon_k}]_{q^{-1}},\label{eq:qg1}\\
E_{-\varepsilon_i+\varepsilon_k}&=-[E_{-\varepsilon_j+\varepsilon_k},E_{-\varepsilon_i+\varepsilon_{j}}]_{q}\label{eq:qg2}.
\end{align}
Similarly we have:
\begin{gather}
 [E_{\varepsilon_j-\varepsilon_k},E_{-\varepsilon_{i}+\varepsilon_{k}} ]=-qK_{jk}^{-1}E_{-\varepsilon_i+\varepsilon_j},\qquad [E_{\varepsilon_{i}-\varepsilon_{k}},E_{-\varepsilon_{j}+\varepsilon_{k}}]=-K_{jk}E_{\varepsilon_{i}-\varepsilon_{j}},\label{eq:qg3}\\
 [E_{\varepsilon_{i}-\varepsilon_{j}},E_{-\varepsilon_{i}+\varepsilon_{k}}]=K_{ij}E_{-\varepsilon_{j}+\varepsilon_{k}},\qquad
 [E_{\varepsilon_{i}-\varepsilon_{k}},E_{-\varepsilon_{i}+\varepsilon_{j}}]=q^{-1}K_{ij}^{-1}E_{\varepsilon_{j}-\varepsilon_{k}}\label{eq:qg4},
\end{gather}
where $K_{ij}=\prod_{k=i}^{j-1}K_k$. Also, relations (3) and (4) of the definition of $U_q(\mathfrak{sl}_{n+1})$ lead to the following:
\begin{gather}
E_{\varepsilon_{j}-\varepsilon_{k}}E_{\varepsilon_i-\varepsilon_k}=q^{-1}E_{\varepsilon_{i}-\varepsilon_{k}}E_{\varepsilon_j-\varepsilon_k},
E_{\varepsilon_{i}-\varepsilon_{k}}E_{\varepsilon_{i}-\varepsilon_{j}}=q^{-1}E_{\varepsilon_{i}-\varepsilon_{j}}E_{\varepsilon_{i}-\varepsilon_{k}},\label{eq:qg6}\\
E_{-\varepsilon_{j}+\varepsilon_{k}}E_{-\varepsilon_i+\varepsilon_k}=q^{-1}E_{-\varepsilon_{i}+\varepsilon_{k}}E_{-\varepsilon_j+\varepsilon_k},
E_{-\varepsilon_{i}+\varepsilon_{k}}E_{-\varepsilon_{i}+\varepsilon_{j}}=q^{-1}E_{-\varepsilon_{i}+\varepsilon_{j}}E_{-\varepsilon_{i}+\varepsilon_{k}}.\label{eq:qg7}
\end{gather}
Finally, if $1\leq i <j <k <l \leq n$ then:
\begin{align}
 [E_{\varepsilon_{i}-\varepsilon_{j}},E_{-\varepsilon_{k}+\varepsilon_{l}}]&=[E_{\varepsilon_{i}-\varepsilon_{j}},E_{\varepsilon_{k}-\varepsilon_{l}}]=[E_{\varepsilon_{i}-\varepsilon_{l}},E_{\varepsilon_{j}-\varepsilon_{k}}]=0\\
 [E_{\varepsilon_{i}-\varepsilon_{k}},E_{\varepsilon_{j}-\varepsilon_{l}}]&=(q-q^{-1})E_{\varepsilon_{i}-\varepsilon_{l}}E_{\varepsilon_{j}-\varepsilon_{k}}\\
 [E_{-\varepsilon_{i}+\varepsilon_{k}},E_{-\varepsilon_{j}+\varepsilon_{l}}]&=(q-q^{-1})E_{-\varepsilon_{i}+\varepsilon_{l}}E_{-\varepsilon_{j}+\varepsilon_{k}}\\
 [E_{\varepsilon_{i}-\varepsilon_{k}},E_{-\varepsilon_{j}+\varepsilon_{l}}]&=-(q-q^{-1})K_{jk}E_{-\varepsilon_{i}+\varepsilon_{j}}E_{\varepsilon_{k}-\varepsilon_{l}}\\
 [E_{\varepsilon_{j}-\varepsilon_{l}},E_{-\varepsilon_{i}+\varepsilon_{k}}]&=(q-q^{-1})K_{jk}^{-1}E_{\varepsilon_{k}-\varepsilon_{l}}E_{-\varepsilon_{i}+\varepsilon_{j}}.
\end{align}
Finally, since the $T_i$ are $U_q(\mathfrak{g})$ automorphisms, we have:
\begin{equation}
[E_{\varepsilon_i-\varepsilon_j},E_{-\varepsilon_i+\varepsilon_j}]=[K_{ij};0]
\end{equation}
where $[K;j]=\frac{q^jK-q^{-j}K^{-1}}{q-q^{-1}}$ for invertible $K\in \mathbb{F}[K_1^{\pm 1},K_2^{\pm 1},\dots, K_n^{\pm 1}]$ and $j\in \mathbb{Z}$.

Let $V$ be a $U_q(\mathfrak{sl}_{n+1})$-module. For $\lambda \in (\mathbb{F}^\times)^n$, the weight space $V_\lambda$ is defined to be the subspace $\{v\in V|K_i\cdot v=\lambda_iv\}$. It is easy to show that the sum of weight spaces in $V$ over all $\lambda\in (\mathbb{F}^\times)^n$ is direct. Moreover, if $V$ is finite-dimensional then it is the sum of its weight spaces (see \cite{CP}) though the same is not necessarily true if $V$ is infinite-dimensional. A $U_q(\mathfrak{sl}_{n+1})$-module that is the direct sum of its weight spaces is called a $U_q(\mathfrak{sl}_{n+1})$ weight module. Throughout this paper, we will consider only irreducible modules in the category of $U_q(\mathfrak{sl}_{n+1})$ weight modules.

\section{Classification of irreducible completely pointed modules}
Let $\mathfrak{g}=\mathfrak{sl}_{n+1}$ and $\Phi=\Phi(\mathfrak{g})$ be the root system of $\mathfrak{g}$. Let $V$ be an irreducible $U_q(\mathfrak{g})$ weight module and $\alpha \in \Phi$. On a weight module, the only possible eigenvalue of $E_\alpha$ is 0, hence $E_\alpha$ either acts nilpotently or injectively on a given weight vector. The subset of vectors on which $E_\alpha$ acts nilpotently (resp. injectively) is a submodule of $V$. Since $V$ is irreducible, we see that $E_\alpha$ acts nilpotently on all of $V$ or else it acts injectively. In the first case, $E_\alpha$ is called \emph{locally nilpotent} and in the second it is called \emph{torsion free}. Highest weight modules are a special case where every positive root vector $E_\alpha$ is locally nilpotent. The other extreme is where every root vector is torsion free. Finally, there are cases where a certain subset of positive root vectors are locally nilpotent but not necessarily all of them. We discuss each case below.

\subsection{Highest weight modules} The irreducible highest weight $U_q(\mathfrak{g})$-module with highest weight $\lambda\in (\mathbb{F}^\times)^n$ is denoted $L(\lambda)$.  Note that for us, unlike the finite-dimensional case for example, $K_i$ can have arbitrary eigenvalues in $\mathbb{F}\backslash\{0\}$, not just powers of $q$. Also this is done so that we can have examples of torsion-free modules (see below).

\begin{lemma} \label{Lem:det}  Assume $V$ is a completely pointed $U_q(\mathfrak{g})$-module and $v\in V$ is a weight vector. For $\theta$ in the positive root lattice, suppose $x_1,x_2\in U_q(\mathfrak{g})_\theta$ and $y_1,y_2\in U_q(\mathfrak{g})_{-\theta}.$  Then $y_ix_j\cdot v=\gamma_{i,j}v$ for some $\gamma_{i,j}\in \mathbb{F}$ and $i,j\in \{1,2\}$, and the $2\times 2$ matrix $(\gamma_{ij})$ is singular.
\end{lemma}

\begin{proof}
Same as the proof in \cite[Lemma 3.2]{BBL}, with $\mathbb{F}$ in place of $\mathbb{C}$.
\end{proof}

\begin{proposition}[Analogous to {\cite[Proposition 3.2]{BBL}}]\label{prop:highest}
The irreducible highest weight $U_q(\mathfrak{sl}_{n+1})$-module $L(\lambda)$ is completely pointed only if $\lambda=\pm 1,\lambda_i=\pm q,\lambda_1=c,\lambda_n=c$, $\lambda_i\lambda_{i+1}=\pm q^{-1}$ for some $i=1,2,\dots, n-1$, where $c\in \mathbb{F^\times}$ is arbitrary and all unspecified entries are $\pm 1$.
\end{proposition}

\begin{proof}
Let $v^+$ be a highest weight vector of $L(\lambda)$, where $\lambda=(\lambda_1,\lambda_2,\dots, \lambda_n)\in (\mathbb{F}^\times)^n$. All irreducible highest weight $U_q(\mathfrak{sl}_2)$-modules are completely pointed, and equal to $L(c)$ for some $c$ which proves the $n=1$ case.

Suppose $n>1$.  Let $x_1=E_{-\varepsilon_{i}+\varepsilon_{i+2}}, x_2=E_{-\varepsilon_i+\varepsilon_{i+1}}E_{-\varepsilon_{i+1}+\varepsilon_{i+2}}\in U_q(\mathfrak{g})_{-\varepsilon_i+\varepsilon_{i+2}}$, and $y_1=E_{\varepsilon_i-\varepsilon_{i+2}},y_2=E_{\varepsilon_i-\varepsilon_{i+1}}E_{\varepsilon_{i+1}-\varepsilon_{i+2}}\in U_q(\mathfrak{g})_{\varepsilon_i-\varepsilon_{i+2}}$ for $i\in \{1,2,\dots, n-1\}$. We denote $\lambda_{ij}=\prod_{k=i}^{j-1}\lambda_k$ and compute:
\begin{align*}
y_1x_1\cdot v_\lambda&=E_{\varepsilon_i-\varepsilon_{i+2}}E_{-\varepsilon_{i}+\varepsilon_{i+2}}\cdot v_\lambda\\
		&=E_{-\varepsilon_{i}+\varepsilon_{i+2}}E_{\varepsilon_i-\varepsilon_{i+2}}\cdot v_\lambda+[K_{i,i+2};0]\cdot v_\lambda\\
		&=[\lambda_{i,i+2};0]v_\lambda,\\
y_1x_2\cdot v_\lambda&=E_{\varepsilon_i-\varepsilon_{i+2}}E_{-\varepsilon_i+\varepsilon_{i+1}}E_{-\varepsilon_{i+1}+\varepsilon_{i+2}}\cdot v_\lambda\\
		&=q^{-1}K_i^{-1}E_{\varepsilon_{i+1}-\varepsilon_{i+2}}E_{-\varepsilon_{i+1}+\varepsilon_{i+2}}\cdot v_\lambda\\
		&=q^{-1}K_i^{-1}[K_{i+1};0]\cdot v_\lambda\\
		&=q^{-1}\lambda_i^{-1}[\lambda_{i+1};0]v_\lambda,\\
y_2x_1\cdot v_\lambda&=E_{\varepsilon_i-\varepsilon_{i+1}}E_{\varepsilon_{i+1}-\varepsilon_{i+2}}E_{-\varepsilon_{i}+\varepsilon_{i+2}}\cdot v_\lambda\\
		&=-E_{\varepsilon_i-\varepsilon_{i+1}}qK_{i+1}^{-1}E_{-\varepsilon_i+\varepsilon_{i+1}}\cdot v_\lambda\\
		&=-K_{i+1}^{-1}[K_{i};0]\cdot v_\lambda\\
		&=-\lambda_{i+1}^{-1}[\lambda_i;0] v_\lambda,\\
y_2x_2\cdot v_\lambda&=E_{\varepsilon_i-\varepsilon_{i+1}}E_{\varepsilon_{i+1}-\varepsilon_{i+2}}E_{-\varepsilon_i+\varepsilon_{i+1}}E_{-\varepsilon_{i+1}+\varepsilon_{i+2}}\cdot v_\lambda\\
		&=[\lambda_{i};0][\lambda_{i+1};0]v_\lambda.
\end{align*}
Therefore, Lemma \ref{Lem:det} gives:
\begin{equation}
[\lambda_{i};0][\lambda_{i+1};0]([\lambda_{i,i+2};0]+q^{-1}\lambda_{i,i+2}^{-1})=0\label{eq:lambda1},
\end{equation}
from which we see $\lambda_i=\pm 1, \lambda_{i+1}=\pm1,$ or $\lambda_i\lambda_{i+1}=\pm q^{-1}$. This finishes the proof for $n<3$, so assume $n\geq 3$.

For $1\leq i < j < k < l \leq n+1$, let $x_1=E_{-\varepsilon_i+\varepsilon_{l}}, x_2=E_{-\varepsilon_i+\varepsilon_{k}}E_{-\varepsilon_{k}+\varepsilon_{l}}\in U_q(\mathfrak{g})_{-\varepsilon_{i}+\varepsilon_{l}}$ and $y_1=E_{\varepsilon_i-\varepsilon_{l}}, y_2=E_{\varepsilon_i-\varepsilon_{j}}E_{\varepsilon_{j}+\varepsilon_{l}}\in U_q(\mathfrak{g})_{\varepsilon_i-\varepsilon_{l}}$.  We compute:
\begin{align*}
y_1x_1\cdot v_\lambda&=E_{\varepsilon_i-\varepsilon_{l}}E_{-\varepsilon_i+\varepsilon_{l}}\cdot v_\lambda\\
		&=[\lambda_{il};0]v_\lambda,\\
y_1x_2\cdot v_\lambda&=E_{\varepsilon_i-\varepsilon_{l}}E_{-\varepsilon_i+\varepsilon_{k}}E_{-\varepsilon_{k}+\varepsilon_{l}}\cdot v_\lambda\\
		&=q^{-1}K_{ik}^{-1}E_{\varepsilon_{k}-\varepsilon_{l}}E_{-\varepsilon_{k}+\varepsilon_{l}}\cdot v_\lambda\\
		&=q^{-1}\lambda_{ik}^{-1}[\lambda_{kl};0]v_\lambda,\\
y_2x_1\cdot v_\lambda&=E_{\varepsilon_i-\varepsilon_{j}}E_{\varepsilon_{j}-\varepsilon_{l}}E_{-\varepsilon_i+\varepsilon_{l}}\cdot v_\lambda\\
		&=-E_{\varepsilon_i-\varepsilon_{j}}qK_{jl}^{-1}E_{-\varepsilon_i+\varepsilon_{j}}\cdot v_\lambda\\
		&=-q^2\lambda_{jl}^{-1}[\lambda_{ij};0] v_\lambda,\\
y_2x_2\cdot v_\lambda &=E_{\varepsilon_i-\varepsilon_{j}}E_{\varepsilon_{j}-\varepsilon_{l}}E_{-\varepsilon_i+\varepsilon_{k}}E_{-\varepsilon_{k}+\varepsilon_{l}}\cdot v_\lambda\\
		&=E_{\varepsilon_i-\varepsilon_{j}}E_{-\varepsilon_i+\varepsilon_{k}}E_{\varepsilon_{j}-\varepsilon_{l}}E_{-\varepsilon_{k}+\varepsilon_{l}}\cdot v_\lambda\\
		&\qquad+E_{\varepsilon_i-\varepsilon_{j}}(q-q^{-1})K_{jk}^{-1}E_{-\varepsilon_i+\varepsilon_{j}}E_{\varepsilon_{k}-\varepsilon_{l}}E_{-\varepsilon_{k}+\varepsilon_{l}}\cdot v_\lambda\\
		&=(q^{2}-1)K_{jk}^{-1}E_{\varepsilon_i-\varepsilon_{j}}E_{-\varepsilon_i+\varepsilon_{j}}E_{\varepsilon_{k}-\varepsilon_{l}}E_{-\varepsilon_{k}+\varepsilon_{l}}\cdot v_\lambda\\
		&=(q^{2}-1)K_{jk}^{-1}[K_{ij};0][K_{kl};0]\cdot v_\lambda\\
		&=(q^{2}-1)\lambda_{jk}^{-1}[\lambda_{ij};0][\lambda_{kl};0]\cdot v_\lambda.
\end{align*}
It follows from Lemma \ref{Lem:det} that 
\begin{equation}
((q^{2}-1)\lambda_{jk}^{-1}[\lambda_{il};0]+q\lambda_{jl}^{-1}\lambda_{ik}^{-1})[\lambda_{ij};0 ][\lambda_{kl};0]=q\lambda_{ij}\lambda_{kl}[\lambda_{ij};0 ][\lambda_{kl};0]=0\label{eq:lambda2}
\end{equation}
Therefore $\lambda_{ij}=\pm 1$ or $\lambda_{kl}=\pm 1$, for $1\leq i < j < k < l \leq n+1$. Let $i$ be minimal such that $\lambda_i\neq \pm 1$. Then we have $\lambda_{j}=\pm 1$ for all $j > i+1$. Since $i$ was chosen to be minimal such that $\lambda_i \neq \pm 1$ the other index $j$ such that $\lambda_j\neq 0$ is $j=i+1$. If $\lambda_{i+1}\neq \pm 1$ the previous paragraph implies that $\lambda_i\lambda_{i+1}=\pm q^{-1}.$

This leaves the case such that only $\lambda_{i}\neq \pm 1$. Let $\lambda_i=c \in \mathbb{F} \backslash \{0\}$ and suppose $1<i<n+1$ (if $i$ is not in that range, then $c$ is not fixed in the statement of the theorem).  Let $x_1=E_{-\varepsilon_{i-1}+\varepsilon_{i+1}}E_{-\varepsilon_{i}+\varepsilon_{i+2}}, x_2=E_{-\varepsilon_i+\varepsilon_{i+1}}E_{-\varepsilon_{i-1}+\varepsilon_{i+2}}, y_1=E_{\varepsilon_{i-1}-\varepsilon_{i+1}}E_{\varepsilon_i-\varepsilon_{i+2}}, y_2=E_{\varepsilon_i-\varepsilon_{i+1}}E_{\varepsilon_{i-1}-\varepsilon_{i+2}}.$  We compute:
\begin{align*}
y_1x_1\cdot v_\lambda&=E_{\varepsilon_{i-1}-\varepsilon_{i+1}}E_{\varepsilon_i-\varepsilon_{i+2}}E_{-\varepsilon_{i-1}+\varepsilon_{i+1}}E_{-\varepsilon_{i}+\varepsilon_{i+2}}\cdot v_\lambda\\
		&=[K_{i-1,i+1};0][K_{i,i+2};0]\cdot v_\lambda\\
		&\qquad+E_{\varepsilon_{i-1}-\varepsilon_{i+1}}K_{i}^{-1}(q-q^{-1})E_{-\varepsilon_{i-1}+\varepsilon_i}E_{\varepsilon_{i+1}-\varepsilon_{i+2}}E_{-\varepsilon_{i}+\varepsilon_{i+2}}\cdot v_\lambda\\
		&=[K_{i-1,i+1};0][K_{i,i+2};0]\cdot v_\lambda-(q^2-1)K_{i}^{-1}E_{\varepsilon_{i-1}-\varepsilon_{i+1}}E_{-\varepsilon_{i-1}+\varepsilon_i}qK_{i+1}^{-1}E_{-\varepsilon_{i}+\varepsilon_{i+1}}\cdot v_\lambda\\
		&=[K_{i-1,i+1};0][K_{i,i+2};0]\cdot v_\lambda-(q-q^{-1})K_{i-1,i+2}^{-1}[K_i;0]\cdot v_\lambda\\
		&=\lambda_{i-1}\lambda_{i+1}([c;0]^2-1+c^{-2})\cdot v_\lambda\\
y_1x_2 \cdot v_\lambda
		&=E_{\varepsilon_{i-1}-\varepsilon_{i+1}}E_{\varepsilon_i-\varepsilon_{i+2}}
		E_{-\varepsilon_i+\varepsilon_{i+1}}E_{-\varepsilon_{i-1}+\varepsilon_{i+2}}
			\cdot v_{\lambda}\\
		&=E_{\varepsilon_{i-1}-\varepsilon_{i+1}}E_{-\varepsilon_i+\varepsilon_{i+1}}E_{\varepsilon_i-\varepsilon_{i+2}}E_{-\varepsilon_{i-1}+\varepsilon_{i+2}}\cdot v_{\lambda}\\
		&\qquad+E_{\varepsilon_{i-1}-\varepsilon_{i+1}}q^{-1}K_i^{-1}E_{\varepsilon_{i+1}-\varepsilon_{i+2}}E_{-\varepsilon_{i-1}+\varepsilon_{i+2}}
			\cdot v_\lambda\\
		&\text{(after several steps, using the fact that $\lambda_{i-1}=\pm 1$)}\\
		&=-K_{i,i+2}^{-1}E_{\varepsilon_{i-1}-\varepsilon_{i+1}}E_{-\varepsilon_{i-1}+\varepsilon_{i+1}}\cdot v_\lambda\\
		&=-K_{i,i+2}^{-1}[K_{i-1,i+1};0]\cdot v_\lambda\\
		&=-\lambda_{i-1}\lambda_{i+1}^{-1}c^{-1}[c;0]v_\lambda\\
y_2x_1\cdot v_\lambda
	&=E_{\varepsilon_i-\varepsilon_{i+1}}E_{\varepsilon_{i-1}-\varepsilon_{i+2}}E_{-\varepsilon_{i-1}+\varepsilon_{i+1}}E_{-\varepsilon_{i}+\varepsilon_{i+2}}\cdot 
	v_\lambda\\
		&=E_{\varepsilon_i-\varepsilon_{i+1}}q^{-1}K_{i-1,i+1}^{-1}E_{\varepsilon_{i+1}-\varepsilon_{i+2}}E_{-\varepsilon_{i}+\varepsilon_{i+2}})\cdot v_\lambda\\
		&=-E_{\varepsilon_i-\varepsilon_{i+1}}K_{i-1,i+2}^{-1}E_{-\varepsilon_i+\varepsilon_{i+1}}\cdot v_\lambda\\
		&=-K_{i-1,i+2}^{-1}E_{\varepsilon_i-\varepsilon_{i+1}}E_{-\varepsilon_i+\varepsilon_{i+1}}\cdot v_\lambda\\
		&=-\lambda_{i-1}^{-1}\lambda_{i+1}^{-1}c^{-1}[c;0]v_\lambda\\
y_2x_2\cdot v_\lambda&=E_{\varepsilon_i-\varepsilon_{i+1}}E_{\varepsilon_{i-1}-\varepsilon_{i+2}}E_{-\varepsilon_i+\varepsilon_{i+1}}E_{-\varepsilon_{i-1}+\varepsilon_{i+2}}\cdot v_{\lambda}\\
	&=E_{\varepsilon_i-\varepsilon_{i+1}}E_{-\varepsilon_i+\varepsilon_{i+1}}E_{\varepsilon_{i-1}-\varepsilon_{i+2}}E_{-\varepsilon_{i-1}+\varepsilon_{i+2}}\cdot v_{\lambda}\\
	&=[K_i;0][K_{i-1,i+2};0]\cdot v_\lambda\\
	&=\lambda_{i-1}\lambda_{i+1}[c;0]^2v_\lambda.
\end{align*}
Therefore, by Lemma \ref{Lem:det} we conclude that $[c;0]^2([c;0]^2-1)=0$.  From this we see that $c=\pm 1$, or $c=\pm q^{\pm1}$, which finishes the proof.
\end{proof}

\emph{Example:} Let $V$ be the natural representation of $U_q(\mathfrak{sl}_{n+1})$. The representations $S^r_q(V),r\in \mathbb{Z}_{\geq 0}$ and $\Lambda^i_q(V), i\in \{1,2,\dots, n\}$ of highest weight $(q^r,1,1,\dots, 1)$ and $(1,1,\dots, 1, q, 1,\dots, 1)$ with $q$ in the $i$th slot, are completely pointed (see \cite{Hayashi} for an explicit construction). They are, up to isomorphism and tensoring with one-dimensional modules, the only finite dimensional representations in our classification (recalling that $L(q^r,1,1,\dots, 1)\cong L(1,1,1,\dots, q^r)$ from the Dynkin diagram symmetry).

\emph{Example:} In this example we take $\mathbb{F}=\mathbb{Q}(q)$. Let $L(q+1)$ be the $U_q(\mathfrak{sl}_2)$-module isomorphic to $U_q(\mathfrak{sl}_2)/J$ where $J$ is the left $U_q(\mathfrak{sl}_2)$ ideal generated by the set $\{E,K-(q+1)\cdot 1\}.$ This is a highest weight $U_q(\mathfrak{sl}_2)$-module with highest weight vector $v_0=1+J.$  As a $\mathbb{Q}(q)$ vector space, $L(q+1)$ has basis $\{v_k=F^{(k)}\cdot v_0 | m\in \mathbb{Z}_{\geq 0}\}$, where $F^{(k)}=F^k/[k]_q!$.  The weight of $v_k$ is given by the following computation:
\begin{equation*}
K\cdot v_k=K\cdot( F^{(k)}\cdot v_0)=q^{-2k}(1+q)v_k.
\end{equation*}
We see that each $v_k$ spans a one-dimensional weight space of weight $q^{-2k}(1+q)$.  Also, $L(q+1)$ is irreducible, as we show by the following standard argument.  Suppose that $L(q+1)$ had a proper submodule $V'$.  Then $V'$ would have a maximal vector of weight $q^{-2k}(q+1)$ for some $k>0$.  This maximal vector would have to be proportional to $v_k$ for $k>0$.  But we have the following:
\begin{eqnarray*}
E\cdot v_k&=&E\cdot (F^{(k)}\cdot v_0)\\
	&=&E F^{(k)}\cdot v_0\\
	&=&\left (F^{(k-1)}\frac{Kq^{-k+1}-K^{-1}q^{k-1}}{q-q^{-1}}-F^{(k)}e\right )\cdot v_0\\
	&=&\frac{(1+q)q^{-k+1}-(1+q)^{-1}q^{k-1}}{q-q^{-1}}F^{(k-1)}\cdot v_0\\
	&=&[1+q;-k+1] v_{k-1}\\
\end{eqnarray*}
which is non-zero when $k>0$.

We now consider the $\mathbb{A}$-form of $L(q+1)$, where $\mathbb{A}=\mathbb{Q}[q,q^{-1}].$  Recall that $U_\mathbb{A}(\mathfrak{sl}_2)$ is defined to be the $\mathbb{A}$-subalgebra of $U_q(\mathfrak{sl}_2)$ generated by the elements $E,F,K,$ and $[K;0]=(K-K^{-1})/(q-q^{-1})$.  The $\mathbb{A}$-form of $L(q+1)$ is now the $U_\mathbb{A}(\mathfrak{sl}_2)$-module $L_\mathbb{A}(q+1)=U_\mathbb{A}(\mathfrak{sl}_2)\cdot v_0.$  We have:
\begin{equation*}
v_0^{(k)}=[K;0]^k \cdot v_0=\left (\frac{K-K^{-1}}{q-q^{-1}}\right )^k\cdot v_0=\left (\frac{(q+1)-(q+1)^{-1}}{q-q^{-1}}\right )^k v_0 
\end{equation*}
are elements of $L_\mathbb{A}(q+1)$ for all $k> 0$ that are not in the $\mathbb{A}$-submodule generated by $v_0$.  However, these elements satisfy the following relations over $\mathbb{A}$:
\begin{eqnarray*}
(q-q^{-1})(q+1) v_0^{(k+1)}=((q+1)^2-1)v_0^{(k)}.
\end{eqnarray*}
Notice in this above example that the problem was not that the $\mathbb{A}$-form in question did not exist--indeed one can consider $U_{\mathbb{A}}(\mathfrak{g})$ acting on any $U_q(\mathfrak{g})$-module. The problem was that passing to the $q=1$ limit provided no information about the $U_q(\mathfrak{g})$-module we wanted to study.

Put differently, if we define the \emph{classical limit} of a highest weight $U_q(\mathfrak{g})$-module $L(\lambda)$ as the $U(\mathfrak{g})$-module $L_{\mathbb{A}}(\lambda)/(q-1)L_{\mathbb{A}}(\lambda)$, then the above equation (with $k=0$) shows that $v_0^{(0)}\in (q-1)L_{\mathbb{A}}(q+1)$. Consequently the classical limit of the $U_q(\mathfrak{sl}_2)$-module $L(q+1)$ is trivial.
The conclusion is that the class of completely pointed $U_q(\mathfrak{g})$-modules is richer than in the classical case, since it consists not only of $q$-deformations of completely pointed $U(\mathfrak{g})$-modules.

\subsection{Torsion free modules} Recall that $U_q(\mathfrak{sl}_2)$ is embedded in $U_q(\mathfrak{gl}_2)$ by adding the invertible element $\bar K_2$ satisfying the relations $\bar K_2 E \bar K_2^{-1}=q^{-1}E$, $\bar K_2 F \bar K_2^{-1}=qF$, and $ K^{\pm}\bar K_2^{\pm}=\bar K_2^{\pm}K^{\pm}$ and defining $\bar K_1^{\pm}=(K\bar K_2^{-1})^{\pm}$.

\begin{lemma}
Let $V$ be an irreducible, torsion free, completely pointed $U_q(\mathfrak{sl}_{2})$-module and $v_\lambda$ a weight vector in $V$.  The action of $U_q(\mathfrak{sl}_{2})$  on $V$ can be extended to a $U_q(\mathfrak{gl}_{2})$ action such that the following relations hold:
\begin{equation*}
FE\cdot v_\lambda=[\bar K_{1};1][\bar K_2;0]\cdot v_\lambda.
\end{equation*}
\end{lemma}

\begin{proof}
Let $c$ be the Casimir element of $U_q(\mathfrak{sl}_2)$:
\begin{equation*}
c=FE+\frac{qK+q^{-1}K^{-1}}{(q-q^{-1})^2}
\end{equation*}
Since $V$ is completely pointed and irreducible, $c$ acts as a scalar $\tau \in \mathbb{F}.$ Therefore, on a weight vector $v_\lambda$ we have:
\begin{equation*}
FE\cdot v_\lambda=\left (\tau-\frac{q\lambda+(q\lambda)^{-1}}{(q-q^{-1})^2}\right )v_\lambda.
\end{equation*}
Since $\mathbb{F}$ is closed under quadratic extensions there are $\mu_1$ and $\mu_2$ satisfying the two equations $\tau=\frac{q\mu_1\mu_2+(q\mu_1\mu_2)^{-1}}{(q-q^{-1})^2}$ and $\lambda=\mu_1\mu_2^{-1}$. This gives:
\begin{align*}
FE\cdot v_\lambda&=\frac{q\mu_1-(q\mu_1)^{-1}}{q-q^{-1}}\left (\frac{\mu_2-\mu_2^{-1}}{q-q^{-1}}\right )v_\lambda.
\end{align*}
We can make $V$ into a $U_q(\mathfrak{gl}_2)$-module by letting $\bar{K}_2$ act as $\mu_2$ on $v_\lambda$ and demanding that the additional relations of $U_q(\mathfrak{gl}_2)$ be satisfied, i.e. $\bar K_2E\bar K_2^{-1}\cdot v_\lambda=q^{-1}E v_\lambda$ and $\bar K_2 F \bar K_2^{-1}\cdot v_\lambda=qF\cdot v_\lambda.$ Then $\bar{K}_1$ acts as $\mu_1$ which gives the desired result.
\end{proof}

As before, $U_q(\mathfrak{sl}_{n+1})$ is embedded in $U_q(\mathfrak{gl}_{n+1})$ by adding the element $\bar{K}_2$ and defining inductively $\bar K_1=K_1\bar K_2$ and $\bar K_{i+1}=K_{i+1}^{-1}\bar{K}_{i}$.

\begin{theorem}\label{th:torsionfree}
Let $V$ be a irreducible, torsion free, completely pointed $U_q(\mathfrak{sl}_{n+1})$-module and $v_\lambda$ a weight vector in $V$.  The action of $U_q(\mathfrak{sl}_{n+1})$  on $V$ can be extended to a $U_q(\mathfrak{gl}_{n+1})$ action such that the following relations hold:
\begin{equation*}
E_{-\varepsilon_i+\varepsilon_j}E_{\varepsilon_i-\varepsilon_j}\cdot v_\lambda=[\bar K_{i};1][\bar K_j;0]\cdot v_\lambda.
\end{equation*}
\end{theorem}
\begin{proof}
If $n=1$ then the result is the previous lemma. So assume $n>1$.

Since $V$ is completely pointed and torsion free we have, for $1\leq i < j <k\leq n+1$, $E_{\varepsilon_{i}-\varepsilon_{j}}E_{\varepsilon_{j}-\varepsilon_{k}}\cdot v_\lambda=\kappa_{ijk} E_{\varepsilon_{i}-\varepsilon_{k}}\cdot v_\lambda,$ for some $\kappa_{ijk} \in \mathbb{F}$. 
Let $z_{ij}\in \mathbb{F}, 1\leq i<j\leq n+1$ be the scalars by which $E_{-\varepsilon_i+\varepsilon_j}E_{\varepsilon_i-\varepsilon_j}$ act on $v_\lambda$.
Recall that $K_{ij}=\prod_{k=i}^{j-1}K_k$ and $\lambda_{ij}=\prod_{k=i}^{j-1}\lambda_k$.

We compute:
\begin{align*}
0&=E_{-\varepsilon_{i}+\varepsilon_{j}}E_{\varepsilon_{i}-\varepsilon_{j}}(E_{\varepsilon_{i}-\varepsilon_{j}}E_{\varepsilon_{j}-\varepsilon_{k}}-\kappa_{ijk} E_{\varepsilon_{i}-\varepsilon_{k}})\cdot v_\lambda\\
	&=E_{-\varepsilon_{i}+\varepsilon_{j}}E_{\varepsilon_{i}-\varepsilon_{j}}(q^{-1}E_{\varepsilon_{j}-\varepsilon_{k}}E_{\varepsilon_{i}-\varepsilon_{j}}-(\kappa_{ijk}+1) E_{\varepsilon_{i}-\varepsilon_{k}})\cdot v_\lambda\\
	&=E_{-\varepsilon_{i}+\varepsilon_{j}}(q^{-2}E_{\varepsilon_{j}-\varepsilon_{k}}E_{\varepsilon_{i}-\varepsilon_{j}}-(q(\kappa_{ijk}+1)+q^{-1}) E_{\varepsilon_{i}-\varepsilon_{k}})E_{\varepsilon_{i}-\varepsilon_{j}}\cdot v_\lambda\\
	&=(-q^{-1}[K_{ij};0]+K_{ij}^{-1}(\kappa_{ijk}+1))E_{\varepsilon_{j}-\varepsilon_{k}}E_{\varepsilon_{i}-\varepsilon_{j}}\cdot v_\lambda\\
	&\qquad+(q^{-2}E_{\varepsilon_{j}-\varepsilon_{k}}E_{\varepsilon_{i}-\varepsilon_{j}}-(q(\kappa_{ijk}+1)+q^{-1}) E_{\varepsilon_{i}-\varepsilon_{k}})z_{ij}\cdot v_\lambda\\
	&=(\kappa_{ijk}+1)(-[\lambda_{ij};1]+(\kappa_{ijk}+1)\lambda_{ij}^{-1})E_{\varepsilon_{i}-\varepsilon_{k}}\cdot v_\lambda-((q-q^{-1})\kappa_{ijk} +q)z_{ij}E_{\varepsilon_{i}-\varepsilon_{k}}\cdot v_\lambda\\
	&=\frac{\phi_{ijk}-q^{-1}}{q-q^{-1}}\left (\frac{\lambda_{ij}^{-1}\phi_{ijk}-q\lambda_{ij}}{q-q^{-1}}\right )E_{\varepsilon_{i}-\varepsilon_{k}}\cdot v_\lambda-\phi_{ijk} z_{ij}E_{\varepsilon_{i}-\varepsilon_{k}}\cdot v_\lambda\tag{*}
\end{align*}
and
\begin{align*}
0&=E_{-\varepsilon_{j}+\varepsilon_{k}}E_{\varepsilon_{j}-\varepsilon_{k}}(E_{\varepsilon_{i}-\varepsilon_{j}}E_{\varepsilon_{j}-\varepsilon_{k}}-\kappa_{ijk} E_{\varepsilon_{i}-\varepsilon_{k}})\cdot v_\lambda\\
	&=E_{-\varepsilon_{j}+\varepsilon_{k}}(qE_{\varepsilon_{i}-\varepsilon_{j}}E_{\varepsilon_{j}-\varepsilon_{k}}+(q-\kappa_{ijk} q^{-1})E_{\varepsilon_{i}-\varepsilon_{k}})E_{\varepsilon_{j}-\varepsilon_{k}}\cdot v_\lambda\\
	&=(-[K_{jk};0]-\kappa_{ijk} q^{-1}K_{jk})E_{\varepsilon_{i}-\varepsilon_{j}}E_{\varepsilon_{j}-\varepsilon_{k}}\cdot v_\lambda\\
	&\qquad+(qE_{\varepsilon_{i}-\varepsilon_{j}}E_{\varepsilon_{j}-\varepsilon_{k}}+(q-\kappa_{ijk} q^{-1})E_{\varepsilon_{i}-\varepsilon_{k}})E_{-\varepsilon_{j}+\varepsilon_{k}}E_{\varepsilon_{j}-\varepsilon_{k}}\cdot v_\lambda\\
	&=\kappa_{ijk}(-[\lambda_{jk};1]-\kappa_{ijk} \lambda_{jk})E_{\varepsilon_{i}-\varepsilon_{k}}\cdot v_\lambda+((q- q^{-1})\kappa_{ijk}+q)z_{jk}E_{\varepsilon_{i}-\varepsilon_{k}}\cdot v_\lambda\\
	&=\frac{\phi_{ijk}-q}{q-q^{-1}}\left (\frac{q^{-1}\lambda_{jk}^{-1}-\phi_{ijk}\lambda_{jk}}{q-q^{-1}}\right )E_{\varepsilon_{i}-\varepsilon_{k}}\cdot v_\lambda+\phi_{ijk} z_{jk}E_{\varepsilon_{i}-\varepsilon_{k}}\cdot v_\lambda\tag{**}.
\end{align*}
where $\phi_{i_1,i_2,i_3}=q+(q- q^{-1})\kappa_{i_i,i_2,i_3}$ for arbitrary $i_1,i_2,i_3$. If $\phi_{ijk}=0$ then (*) gives the following:
\begin{equation*}
0=\frac{\lambda_{ij}}{(q-q^{-1})^2}
\end{equation*}
which is a contradiction (since every $K_i$ acts as an invertible scalar). Therefore, $\phi_{ijk}\neq 0$. In addition, we have:
\begin{align*}
z_{ij}z_{jk}v_{\lambda}&=E_{\varepsilon_j-\varepsilon_i}E_{\varepsilon_i-\varepsilon_j}E_{\varepsilon_k-\varepsilon_j}E_{\varepsilon_j-\varepsilon_k}\cdot v_\lambda\\
	&=E_{\varepsilon_j-\varepsilon_i}E_{\varepsilon_k-\varepsilon_j}E_{\varepsilon_i-\varepsilon_j}E_{\varepsilon_j-\varepsilon_k}\cdot v_\lambda\\
	&=\kappa_{ijk}E_{\varepsilon_j-\varepsilon_i}E_{\varepsilon_k-\varepsilon_j}E_{\varepsilon_i-\varepsilon_k}\cdot v_\lambda\\
	&=(q^{-1}\kappa_{ijk}E_{\varepsilon_k-\varepsilon_j}E_{\varepsilon_j-\varepsilon_i}E_{\varepsilon_i-\varepsilon_k}+q^{-1}\kappa_{ijk}E_{\varepsilon_k-\varepsilon_i}E_{\varepsilon_i-\varepsilon_k})\cdot v_\lambda\\
	&=q^{-1}\kappa_{ijk}(\kappa_{ijk}z_{ik}-(q-q^{-1})z_{ij}z_{jk})v_\lambda+q^{-1}\kappa_{ijk}z_{ik}v_\lambda\\
\end{align*}
whence
\begin{equation*}
\kappa_{ijk}(\kappa_{ijk}+1)z_{ik}=\phi_{ijk}z_{ij}z_{jk}. \tag{***}
\end{equation*}
From (*), (**), and (***) we deduce:
\begin{align}
z_{ij}&=\phi_{ijk}^{-1}\left (\frac{\phi_{ijk}-q^{-1}}{q-q^{-1}}\right )\left (\frac{\lambda_{ij}^{-1}\phi_{ijk}-q\lambda_{ij}}{q-q^{-1}}\right )\label{eq:z1}\\
z_{jk}&=-\phi_{ijk}^{-1}\left (\frac{\phi_{ijk}-q}{q-q^{-1}}\right )\left (\frac{q^{-1}\lambda_{jk}^{-1}-\phi_{ijk}\lambda_{jk}}{q-q^{-1}}\right )\label{eq:z2}\\
z_{ik}&=-\phi_{ijk}^{-1}\left (\frac{\lambda_{ij}^{-1}\phi_{ijk}-q\lambda_{ij}}{q-q^{-1}}\right )\left (\frac{q^{-1}\lambda_{jk}^{-1}-\phi_{ijk}\lambda_{jk}}{q-q^{-1}}\right )\label{eq:z3}.
\end{align}
This covers the case where $n=2$, so suppose $n>2$. If $1\leq i<j<k<l\leq n+1$, then \eqref{eq:z2} and \eqref{eq:z3} give:
\begin{multline}
-\phi_{ijl}^{-1}\left (\frac{\phi_{ijl}-q}{q-q^{-1}}\right )\left (\frac{q^{-1}\lambda_{jl}^{-1}-\phi_{ijl}\lambda_{jl}}{q-q^{-1}}\right )\\
=-\phi_{jkl}^{-1}\left (\frac{\lambda_{jk}^{-1}\phi_{jkl}-q\lambda_{jk}}{q-q^{-1}}\right )\left (\frac{q^{-1}\lambda_{kl}^{-1}-\phi_{jkl}\lambda_{kl}}{q-q^{-1}}\right )
\end{multline}
which yields:
\begin{equation}
(\phi_{jkl}\phi_{ijl}-\lambda_{kl}^{-2})(\lambda_{jk}^2\phi_{ijl}-\phi_{jkl})=0\label{eq:phi1}.
\end{equation}
A similar argument using \eqref{eq:z1} and \eqref{eq:z3} gives:
\begin{equation}
(\phi_{ikl}\phi_{ijk}-\lambda_{ij}^{2})(\lambda_{jk}^2\phi_{ijk}-\phi_{ikl})=0\label{eq:phi2}.
\end{equation}

We compute:
\begin{align*}
\kappa_{jkl}E_{\varepsilon_j-\varepsilon_l}\cdot v_\lambda &=E_{\varepsilon_j-\varepsilon_k}E_{\varepsilon_k-\varepsilon_l}\cdot v_\lambda\\
&=(-E_{\varepsilon_j-\varepsilon_l}+q^{-1}E_{\varepsilon_k-\varepsilon_l}E_{\varepsilon_j-\varepsilon_k})\cdot v_\lambda
\end{align*}
hence $q(\kappa_{jkl}+1)E_{\varepsilon_j-\varepsilon_l}\cdot v_\lambda=E_{\varepsilon_k-\varepsilon_l}E_{\varepsilon_j-\varepsilon_k}\cdot v_\lambda$. Therefore we have:
\begin{align*}
q\kappa_{ijl}(\kappa_{jkl}+1)E_{\varepsilon_i-\varepsilon_l}\cdot v_\lambda
&=q(\kappa_{jkl}+1)E_{\varepsilon_i-\varepsilon_j}E_{\varepsilon_j-\varepsilon_l}\cdot v_\lambda\\
&=E_{\varepsilon_k-\varepsilon_l}E_{\varepsilon_i-\varepsilon_j}E_{\varepsilon_j-\varepsilon_k}\cdot v_\lambda\\
&=\kappa_{ijk}E_{\varepsilon_k-\varepsilon_l}E_{\varepsilon_i-\varepsilon_k}\cdot v_\lambda\\
&=q\kappa_{ijk}(\kappa_{ikl}+1)E_{\varepsilon_i-\varepsilon_l}\cdot v_\lambda.
\end{align*}
Therefore, since $E_{\varepsilon_i-\varepsilon_l}\cdot v_\lambda\neq 0$, we have:
\begin{equation}
\kappa_{ijl}(\kappa_{jkl}+1)=\kappa_{ijk}(\kappa_{ikl}+1)\label{eq:kappa1}
\end{equation}
and analogously:
\begin{equation}
\kappa_{jil}(\kappa_{ikl}+1)=\kappa_{jik}(\kappa_{jkl}+1).\label{eq:kappa1p}
\end{equation}

Now we derive more relations for $\kappa_{i_1,i_2,i_3}$:
\begin{align*}
E_{\varepsilon_i-\varepsilon_j}E_{\varepsilon_j-\varepsilon_l}E_{\varepsilon_j-\varepsilon_k}\cdot v_\lambda&=(-E_{\varepsilon_i-\varepsilon_l}E_{\varepsilon_j-\varepsilon_k}+q^{-1}E_{\varepsilon_j-\varepsilon_l}E_{\varepsilon_i-\varepsilon_j}E_{\varepsilon_j-\varepsilon_k})\cdot v_\lambda\\
	&=(-E_{\varepsilon_i-\varepsilon_l}E_{\varepsilon_j-\varepsilon_k}+q^{-1}\kappa_{ijk}E_{\varepsilon_j-\varepsilon_l}E_{\varepsilon_i-\varepsilon_k})\cdot v_\lambda\\
	&=(-q^{-1}(q-q^{-1})\kappa_{ijk}-1)E_{\varepsilon_i-\varepsilon_l}E_{\varepsilon_j-\varepsilon_k}\cdot v_\lambda\\
	&\qquad +q^{-1}\kappa_{ijk}E_{\varepsilon_i-\varepsilon_k}E_{\varepsilon_j-\varepsilon_l}\cdot v_\lambda
\end{align*}
and
\begin{align*}
E_{\varepsilon_i-\varepsilon_j}E_{\varepsilon_j-\varepsilon_l}E_{\varepsilon_j-\varepsilon_k}\cdot v_\lambda&=q^{-1}E_{\varepsilon_i-\varepsilon_j}E_{\varepsilon_j-\varepsilon_k}E_{\varepsilon_j-\varepsilon_l}\cdot v_\lambda\\
	&=(-q^{-1}E_{\varepsilon_i-\varepsilon_k}E_{\varepsilon_j-\varepsilon_l}+q^{-2}E_{\varepsilon_j-\varepsilon_k}E_{\varepsilon_i-\varepsilon_j}E_{\varepsilon_j-\varepsilon_l})\cdot v_\lambda\\
	&=(-q^{-1}E_{\varepsilon_i-\varepsilon_k}E_{\varepsilon_j-\varepsilon_l}+q^{-2}\kappa_{ijl}E_{\varepsilon_j-\varepsilon_k}E_{\varepsilon_i-\varepsilon_l})\cdot v_\lambda.
\end{align*}
Therefore:
\begin{align*}
0&=\kappa_{jkl}(q^{-1}(q-q^{-1})\kappa_{ijk}+1+q^{-2}\kappa_{ijl})E_{\varepsilon_i-\varepsilon_l}E_{\varepsilon_j-\varepsilon_k}\cdot v_\lambda\\
	&\qquad-q^{-1}\kappa_{jkl}(\kappa_{ijk}+1)E_{\varepsilon_i-\varepsilon_k}E_{\varepsilon_j-\varepsilon_l}\cdot v_\lambda\\
	&=\kappa_{jkl}(q^{-1}(q-q^{-1})\kappa_{ijk}+1+q^{-2}\kappa_{ijl})E_{\varepsilon_i-\varepsilon_l}E_{\varepsilon_j-\varepsilon_k}\cdot v_\lambda\\
	&\qquad -q^{-1}(\kappa_{ijk}+1)E_{\varepsilon_i-\varepsilon_k}E_{\varepsilon_j-\varepsilon_k}E_{\varepsilon_k-\varepsilon_l}\cdot v_\lambda\\
	&=\kappa_{jkl}(q^{-1}(q-q^{-1})\kappa_{ijk}+1+q^{-2}\kappa_{ijl})E_{\varepsilon_i-\varepsilon_l}E_{\varepsilon_j-\varepsilon_k}\cdot v_\lambda\\
		&\qquad -\kappa_{ikl}(\kappa_{ijk}+1)E_{\varepsilon_i-\varepsilon_l}E_{\varepsilon_j-\varepsilon_k}\cdot v_\lambda
\end{align*}
which gives:
\begin{equation}
\kappa_{jkl}(q^{-1}(q-q^{-1})\kappa_{ijk}+1+q^{-2}\kappa_{ijl})=\kappa_{ikl}(\kappa_{ijk}+1).
\end{equation}
Subtracting \eqref{eq:kappa1} from the above we see:
\begin{align}
\kappa_{jkl}(q^{-1}(q-q^{-1})\kappa_{ijk}+1-q^{-1}(q-q^{-1})\kappa_{ijl})-\kappa_{ijl}&=\kappa_{ikl}-\kappa_{ijk}\\
(q^{-1}\phi_{jkl}-1)(\phi_{ijk}-\phi_{ijl})+\phi_{jkl}-\phi_{ijl}&=\phi_{ikl}-\phi_{ijk}\\
q^{-1}\phi_{jkl}(\phi_{ijk}-\phi_{ijl})+\phi_{jkl}-\phi_{ikl}&=0.\label{eq:kappa3}
\end{align}

We repeat the same argument with the indices $i$ and $j$ transposed to obtain:
\begin{align*}
E_{\varepsilon_j-\varepsilon_i}E_{\varepsilon_i-\varepsilon_l}E_{\varepsilon_i-\varepsilon_k}\cdot v_\lambda&=(-q^{-1}K_{ij}^{-1}E_{\varepsilon_j-\varepsilon_l}E_{\varepsilon_i-\varepsilon_k}+E_{\varepsilon_i-\varepsilon_l}E_{\varepsilon_j-\varepsilon_i}E_{\varepsilon_i-\varepsilon_k})\cdot v_\lambda\\
	&=(-q^{-1}\lambda_{ij}^{-1}E_{\varepsilon_j-\varepsilon_l}E_{\varepsilon_i-\varepsilon_k}+\kappa_{jik}E_{\varepsilon_i-\varepsilon_l}E_{\varepsilon_j-\varepsilon_k})\cdot v_\lambda\\
	&=-q^{-1}\lambda_{ij}^{-1}E_{\varepsilon_i-\varepsilon_k}E_{\varepsilon_j-\varepsilon_l}\cdot v_\lambda\\
	&\qquad+(\kappa_{jik}+q^{-1}(q-q^{-1})\lambda_{ij}^{-1})E_{\varepsilon_i-\varepsilon_l}E_{\varepsilon_j-\varepsilon_k}\cdot v_\lambda
\end{align*}
and
\begin{align*}
E_{\varepsilon_j-\varepsilon_i}E_{\varepsilon_i-\varepsilon_l}E_{\varepsilon_i-\varepsilon_k}\cdot v_\lambda&=q^{-1}E_{\varepsilon_j-\varepsilon_i}E_{\varepsilon_i-\varepsilon_k}E_{\varepsilon_i-\varepsilon_l}\cdot v_\lambda\\
	&=(-q^{-2}K_{ij}^{-1}E_{\varepsilon_j-\varepsilon_k}E_{\varepsilon_i-\varepsilon_l}+q^{-1}E_{\varepsilon_i-\varepsilon_k}E_{\varepsilon_j-\varepsilon_i}E_{\varepsilon_i-\varepsilon_l})\cdot v_\lambda\\
	&=(-q^{-2}\lambda_{ij}^{-1}E_{\varepsilon_i-\varepsilon_l}E_{\varepsilon_j-\varepsilon_k}+q^{-1}\kappa_{jil}E_{\varepsilon_i-\varepsilon_k}E_{\varepsilon_j-\varepsilon_l})\cdot v_\lambda.
\end{align*}
Hence:
\begin{align*}
0&=\kappa_{ikl}(\kappa_{jik}+\lambda_{ij}^{-1})E_{\varepsilon_i-\varepsilon_l}E_{\varepsilon_j-\varepsilon_k}\cdot v_\lambda-q^{-1}\kappa_{ikl}(\kappa_{jil}+\lambda_{ij}^{-1})E_{\varepsilon_i-\varepsilon_k}E_{\varepsilon_j-\varepsilon_l}\cdot v_\lambda\\
	&=(\kappa_{jik}+\lambda_{ij}^{-1})E_{\varepsilon_j-\varepsilon_k}E_{\varepsilon_i-\varepsilon_k}E_{\varepsilon_k-\varepsilon_l}\cdot v_\lambda-q^{-1}\kappa_{ikl}(\kappa_{jil}+\lambda_{ij}^{-1})E_{\varepsilon_i-\varepsilon_k}E_{\varepsilon_j-\varepsilon_l}\cdot v_\lambda\\
	&=q^{-1}\kappa_{jkl}(\kappa_{jik}+\lambda_{ij}^{-1})E_{\varepsilon_i-\varepsilon_k}E_{\varepsilon_j-\varepsilon_l}\cdot v_\lambda-q^{-1}\kappa_{ikl}(\kappa_{jil}+\lambda_{ij}^{-1})E_{\varepsilon_i-\varepsilon_k}E_{\varepsilon_j-\varepsilon_l}\cdot v_\lambda\\
\end{align*}
which gives:
\begin{equation}
\kappa_{jkl}(\kappa_{jik}+\lambda_{ij}^{-1})=\kappa_{ikl}(\kappa_{jil}+\lambda_{ij}^{-1}).\label{eq:kappa4}
\end{equation}
Subtracting \eqref{eq:kappa1p} from the above gives:
\begin{align}
\lambda_{ij}^{-1}\kappa_{jkl}-\kappa_{jik}=\lambda_{ij}^{-1}\kappa_{ikl}-\kappa_{jil}\\
\lambda_{ij}^{-1}(\kappa_{jkl}-\kappa_{ikl})=\kappa_{jik}-\kappa_{jil}.\label{eq:kappa5}
\end{align}
We compute:
\begin{align*}
0&=(E_{\varepsilon_k-\varepsilon_j}E_{\varepsilon_j-\varepsilon_i}E_{\varepsilon_i-\varepsilon_k}-\kappa_{jik}E_{\varepsilon_k-\varepsilon_j}E_{\varepsilon_j-\varepsilon_k})\cdot v_\lambda\\
	&=(\kappa_{ijk}z_{ik}-(q-q^{-1})z_{ij}z_{jk}-\kappa_{jik}z_{jk})v_\lambda\\
	&=((\kappa_{ijk}+1)^{-1}\phi_{ijk}-(q-q^{-1}))z_{ij}z_{jk}v_\lambda-\kappa_{jik}z_{jk}v_\lambda.
\end{align*}
Hence, using \eqref{eq:z1}--\eqref{eq:z3} we see:
\begin{equation}
\kappa_{jik}=\frac{q^{-1}\lambda_{ij}^{-1}-\lambda_{ij}\phi_{ijk}^{-1}}{q-q^{-1}}
\end{equation}
and similarly,
\begin{equation}
\kappa_{jil}=\frac{q^{-1}\lambda_{ij}^{-1}-\lambda_{ij}\phi_{ijl}^{-1}}{q-q^{-1}}.
\end{equation}
Using the above in \eqref{eq:kappa5} gives:
\begin{equation}
\lambda_{ij}^{-1}\phi_{ijk}\phi_{ijl}(\phi_{jkl}-\phi_{ikl})=\lambda_{ij}(\phi_{ijk}-\phi_{ijl})\label{eq:kappa6}.
\end{equation}

Equations  \eqref{eq:phi1}, \eqref{eq:phi2}, \eqref{eq:kappa3}, and \eqref{eq:kappa6} have a unique simultaneous solution for $\phi_{ijl}, \phi_{ikl}, \phi_{jkl}$ in terms of $\phi_{ijk}$, namely:
\begin{gather}
\phi_{ijl}=\phi_{ijk}, \phi_{ikl}=\phi_{jkl}=\lambda_{jk}^2\phi_{ijk}\label{eq:main}.
\end{gather}
From these relations and equations \eqref{eq:z1}--\eqref{eq:z3} we see that all the $z_{ij}$ are determined by $\phi_{123}$. Using that $\mathbb{F}$ is closed under quadratic extensions, we choose $\mu_2=\pm(q\phi_{123})^{-1/2}$ and obtain the desired result: $z_{ij}=[\mu_i;1][\mu_j;0]$. 
\end{proof}
We remark that analogues of equations \eqref{eq:kappa3} and \eqref{eq:kappa6} were found in \cite{BriLem1987} and in the $q=1$ limit they imply the first two identities in \eqref{eq:main}, but in our case all four equations are needed.

\subsection{Modules with torsion}
We now set out to prove an extension of the Theorem \ref{th:torsionfree} to the torsion case. Before we do so, let us introduce some notation. Let $V$ be an irreducible $U_q(\mathfrak{sl}_{n+1})$-module. Let $N=\{\beta\in \Phi| \forall v\in V,\exists k>0 \text{ such that } E_\beta^k \cdot v=0\}, T=\{\beta\in \Phi|\forall v\in V, E_{\beta} \cdot v\neq 0\},N_s=N\cap (-N),T_s=T\cap (-T),N_a=N\backslash N_s$ and $T_a=T\backslash T_s$. Finally, define $V^+=\{v\in V| \forall \beta\in N_a\cup N_s^+ ,E_\beta \cdot v =0\}$ where $N_s^+=\Phi^+\cap N_s$. Using that $q$ is not a root of unity it is easy to show that $\Phi=N\cup T$. The following is an analogue of (4.6) and (4.12) in \cite{F}, but in our proof for completely-pointed modules we avoid working with the center of $U_q(\mathfrak{g})$.

\begin{proposition}\label{prop:closed}
Let $V$ be an irreducible completely-pointed $U_q(\mathfrak{sl}_{n+1})$ weight module. Then $N$ and $T$ are closed subsets of $\Phi$.
\end{proposition}
\begin{proof}
Let $\alpha,\beta\in N$ be such that $\alpha+\beta\in \Phi$. Then there exist $0\neq v^+\in V$ such that $E_\alpha\cdot v^+=0$ and $s\in \mathbb{Z}_{>0}$ such that $E_\beta^s\cdot v^+=0$ and $E_\beta^{s-1}\cdot v^+\neq 0$. Note that equations \eqref{eq:qg1}--\eqref{eq:qg4} imply $KE_{\alpha+\beta}=\pm (q^jE_\alpha E_\beta-q^kE_\beta E_\alpha)$ for some $j,k\in \{-1,0,1\},$ and invertible $K\in \mathbb{F}[K_1^{\pm 1}, K_2^{\pm 1},\dots, K_n^{\pm 1}]$. Using this, we compute:
\begin{align}
0&=E_\alpha E_\beta^s\cdot v^+=\pm\sum_{i=0}^{s-1} q^{k_i}E_\beta^i K E_{\alpha+\beta} E_\beta^{s-i-1}\cdot v^+\nonumber \\
	&=\pm\left (\sum_{i=0}^{s-1} q^{k_i+r_i}\right ) K E_{\alpha+\beta}\cdot ( E_\beta^{s-1}\cdot v^+)\label{eq:local_nil}
\end{align}
where $k_i,r_i\in \mathbb{Z},i\in \{1,2,\dots, n\}$ are increasing or decreasing sequences. Therefore $\alpha+\beta \in N$, which gives that $N$ is closed.

Now let $\alpha,\beta\in T$ be such that $\alpha+\beta\in \Phi$. Since $\alpha,\beta\in T$ we have $E_\alpha^k\cdot v_\lambda\neq 0$ and $E_\beta^r \cdot v_\lambda\neq 0$ for all $k,r\in \mathbb{Z}_{\geq 0}$ and all weights $\lambda\in(\mathbb{F}^\times)^{n+1}$, hence $q^{\mathbb{Z}_{\geq 0}(\alpha+\beta)}\text{supp}(V) \subseteq \text{supp}(V)$. If $E_{\alpha+\beta}\in N$ then, given $\lambda\in \text{wt}(V)$, we have an infinite sequence of vectors $v_i\in V_{q^{m_i(\alpha+\beta)}\lambda}$ where $m_i\in \mathbb{Z}_{\geq 0}$ is an increasing sequence such that $E_{\alpha+\beta}\cdot v_i=0$. We compute:
\begin{align*}
0&=E_{\alpha+\beta}\cdot v_i\\
	&=cE_{\alpha+\beta}(E_{-\alpha-\beta})^{m_{i+1}-m_i}\cdot v_{i+1}\\
	&\qquad \text{ (for some $c\in \mathbb{F}\backslash \{0\}$ since $V$ is completely pointed)}\\
	&=c\left(\sum_{i=0}^{m_{i+1}-m_i-1}(E_{-\alpha-\beta})^{i}[K_{\alpha+\beta};0](E_{-\alpha-\beta})^{m_{i+1}-m_i-i}\right )\cdot v_{i+1}\\
	&=c[m_{i+1}-m_i]_q[K_{\alpha+\beta};m_{i+1}-m_i-1](E_{-\alpha-\beta})^{m_{i+1}-m_i-1}\cdot v_{i+1}.
\end{align*}
Since $c(E_{-\alpha-\beta})^{m_{i+1}-m_i-1}\cdot v_{i+1}\neq 0$ and $m_{i+1}>m_i$ we see that $[K_{\alpha+\beta};m_{i+1}-m_i-1]\cdot v_{i+1}=0$. Therefore, letting $\lambda_{\alpha+\beta}\in \mathbb{F}\backslash \{0\}$ be the eigenvalue of $K_{\alpha+\beta}$ on $v_{i+1}$, we see that $\lambda_{\alpha+\beta}=\pm q^{-(m_{i+1}-m_i-1)}$---a non-positive integral power of $q$ times $\pm 1$. However, since $K_{\alpha+\beta}\cdot v_{i+j}=q^{2(m_j-m_{i+1})}\lambda_{\alpha+\beta}v_{i+j}$ and $m_j$ is an increasing sequence of integers, for some $j^*>i$ we must have $K_{\alpha+\beta}$ acting as a non-negative power of $q$ times $\pm 1$ on $v_{j^*}$. This $j^*$ is therefore the highest index in the sequence of $m_j$, contrary to it being an infinite sequence. Therefore $\alpha+\beta\notin N$ and must be in $T$ since $\Phi=N\cup T$.
\end{proof}

As a corollary, we have that $N_s$ and $T_s$ are root subsystems of $\Phi$. 

\begin{lemma}\label{lem:roots}
If $\mathfrak{g}$ is simply-laced, then there exists a base $B$ of $\Phi(\mathfrak{g})$ such that $N_a\subseteq \Phi_B^+$, and every $\alpha\in B\backslash N_a$ is a positive root (with respect to the usual base of $\Phi$).
\end{lemma}
\begin{proof}
Lemma 4.7 (\emph{i}) of \cite{BBL} proves the existence of a base $B$ of $\Phi=\Phi(\mathfrak{g})$ such that $N_a^+\subseteq \Phi_B^+$. We may apply their result since the proof only uses results on root subsystems that satisfy the same hypotheses as in our case. We show how to choose a new base $B_n$ satisfying the same condition, but with every $\alpha \in B \backslash N_a$ positive with respect to the usual base, $B_u$. By the previous proposition we see that $N_s$ ant $T_s$ are a root subsystems of $\Phi$. Let $W_N$ and $W_T$ be the Weyl group of the root subsystems $N_s$ and $T_s$ respectively. We may choose a base $B_p$ of $N_s \cup T_s$ such that $(N_s\cup T_s)\cap B_p$ is contained in $\Phi_{B_u}^+$. It is a well known fact of finite root systems (see \cite[Section 10.1]{H}) that the Weyl group permutes bases. Let $w\in W_T \times W_N$ be the Weyl group element taking $B\cap (N_s \cup T_s)$ to $B_p$. We want to show that $w$ preserves $N_a$. Let $\alpha \in N_a$ and $\beta \in B\cap (N_s \cup T_s)$ be given. We have $r_\beta(\alpha)=\alpha - \langle \beta, \alpha \rangle \beta\in\Phi$. Since $\mathfrak{g}$ is simply laced, $\langle \beta, \alpha \rangle=0$ or $-1$. In the first case, $r_\beta(\alpha)=\alpha\in N_a$. Otherwise, $r_\beta(\alpha)=\alpha+\beta$. By (\emph{ii}) of the above lemma, we have $\alpha+\beta\in N_a$. Hence, $w(\alpha)\in N_a$, and $w(B)$ is the desired basis.
\end{proof}

We let $U_q(\mathfrak{g}_{N_s})$ (resp. $U_q(\mathfrak{n}_{N_s})$) denote the subalgebra of $U_q(\mathfrak{g})$ generated by $\{E_\alpha|\alpha\in N_s\}$ and  (resp. $\{E_\alpha|\alpha\in N_s^+\}$).
\begin{proposition}
Suppose $V$ is an irreducible completely-pointed $U_q(\mathfrak{g})$-module. Then the following hold:
\begin{enumerate}
\item if $N\neq \varnothing$ then the set $V^+$ is non-zero, i.e. there exists $v^+\in V\backslash \{0\}$ such that $E_\alpha\cdot v^+=0$ for $\alpha\in N_s^+\cup N_a$,
\item either $B\cap T$ is empty or it corresponds to a connected part of the Dynkin diagram of $\Phi_B$ where $B$ is the basis of the previous lemma.
\end{enumerate}
\end{proposition}
\begin{proof}
Proof of (1). Let $v^*\in V\backslash\{0\}$ and $\{\beta_1,\beta_2,\dots, \beta_l\}=N\cap B$ where $B$ is the base of $\Phi$ given in the previous lemma. Then there exist $r_j\in \mathbb{Z}_{>0}, j\in \{1,2,\dots, l\}$ such that $E_{\beta_j}^{r_j}\prod_{k=j+1}^{l}E_{\beta_{k}}^{r_k-1}\cdot v^*=0$ and $r_j$ is minimal that this occurs. Let $v^+=\prod_{k=1}^{l}E_{\beta_{k}}^{r_k-1}\cdot v^*.$ Then $v^+\neq 0$ and $E_{\beta_j}\cdot v^+=0,j\in\{1,2,\dots, l\}$. Since $V$ is completely pointed, we have $E_\alpha E_\beta\cdot v_\lambda=cE_\beta E_\alpha\cdot v_\lambda$ for some $c\in \mathbb{F}$ whenever $v_\lambda$ is a weight vector such that $E_\beta\cdot v_\lambda\neq 0$, hence $E_{\beta_k}\cdot v^+=0$ for all $k\in\{1,2,\dots, l\}$. Now, let $\gamma\in N_a\cup N_s^+$. Then $\gamma\in \Phi_B^+$ hence it can be written as a positive integral combination of elements of $B$. At least one of these simple roots must be in $N\cap B$ since otherwise $\gamma \in T$ which is a contradiction. Let $\alpha,\beta \in B$ such that $E_\beta \cdot v^+=0$ and $\alpha+\beta\in \Phi$. We compute 
\begin{align*}
E_{\alpha+\beta} \cdot v^+&=\pm K(q^jE_{\beta}E_{\alpha}-q^{k}E_{\alpha}E_\beta)\cdot v^+\\
	&=\pm K(q^j-cq^{k})E_{\beta}E_\alpha\cdot v^+\\
	&=0
\end{align*}
for some $j,k\in \{-1,0,1\}, K\in \mathbb{F}[K_1^{\pm 1}, K_2^{\pm 1},\dots, K_n^{\pm 1}]$ and $c\in \mathbb{F}$. Let $\beta_i\in N\cap B$ be one of the simple roots in the decomposition of $\gamma$. Then there exists a sequence of roots $\gamma_0=\beta_i,\gamma_1,\dots, \gamma_r=\gamma$ with $\gamma_{i+1}-\gamma_i\in B$. Since $E_{\beta_i}\cdot v^+=0$ we may use induction on the sequence of $\gamma_k$ to see that $E_\gamma\cdot v^+=0$. Therefore $v^+$ is the vector we are looking for.

The proof of (2) is similar to Lemma 4.9 of \cite{BBL} and we leave it to the reader.
\end{proof}

We are now ready to prove:
\begin{mainthmi}
Let $V$ be an irreducible, infinite-dimensional, completely-pointed $U_q(\mathfrak{sl}_{n+1})$-module and let $v^+\in V^+$ be given. Then the action of $U_q(\mathfrak{sl}_{n+1})$  on $V$ can be extended to a $U_q(\mathfrak{gl}_{n+1})$ action such that the following relations hold:
\begin{equation*}
E_{-\varepsilon_i+\varepsilon_j}E_{\varepsilon_i-\varepsilon_j}\cdot v^+=[\bar K_{i};1][\bar K_j;0]\cdot v^+.
\end{equation*}
Also, we have
\begin{equation*}
F_{i}E_{i}\cdot v_\lambda=[\bar K_{i};1][\bar K_{i+1};0]\cdot v_\lambda
\end{equation*}
for any weight vector $v_{\lambda}\in V$.
\end{mainthmi}
\begin{proof}There are three cases: the extreme cases $T_s=\varnothing,T_s=\Phi$ and the intermediate case $\varnothing \subsetneq T_s \subsetneq \Phi.$

\emph{Case i}: $T_s=\varnothing$. In this case, $V$ is a highest weight module, hence, by Proposition \ref{prop:highest}, is isomorphic to $L(\lambda)$ where $\lambda=(c,\pm 1, \pm 1, \cdots, \pm 1),(\pm 1,\pm 1, \cdots, \pm 1,c),$ with $c$ not a positive integer power of $q$, or $(\pm 1, \pm 1, \cdots, c, c^{-1}q^{-1}, \cdots , \pm 1)$. As before, let $z_{ij}\in \mathbb{F}, 1\leq i<j\leq n+1$ be the scalars by which $E_{-\varepsilon_i+\varepsilon_j}E_{\varepsilon_i-\varepsilon_j}$ act on $v_\lambda$. We have $z_{ij}=0$ for $i<j\in \{1,2,\dots, n+1\}$. In the first case, we may choose $\mu_1=c$, and $\mu_i=\pm 1, 1 < i \leq n+1$. In the second case $\mu_{n+1}=q^{-1}c^{-1}$ and $\mu_i=\pm q^{-1}, 1\leq i < n+1$ gives the desired result. In the third case, if $\lambda_k=c$ then we choose $\mu_i=\pm q^{-1}, 1\leq i < k, \mu_k=q^{-1}c^{-1},$ and $\mu_i=\pm 1,k<i\leq n+1$ satisfies the conditions of the theorem.

\emph{Case ii:} $T_s=\Phi$. In this case, $V$ is a torsion free module. By Theorem \ref{th:torsionfree} the conclusion holds for some weight vector $v_\lambda$. As in the last part of the next case one shows that then it holds for any weight vector.

\emph{Case iii:} $\varnothing \subsetneq T_s \subsetneq \Phi.$ Let $\beta_1,\beta_2,\dots, \beta_n$ be a base of $\Phi$ such that $N_a\subseteq \Phi_B^+$ with $\beta_i \in T_s$ for $i\in \{k,k+1,\dots,l\}$. Let $v^+$ be an invariant vector of weight $\lambda$ for $E_\alpha,\alpha\in N_a\cup  N_s^+$. Using Theorem \ref{th:torsionfree} we can choose $\mu_{s_i}, i\in \{k,k+1,\dots, l+1\}$ such that $z_{s_i,s_j}=[\mu_{s_i};1][\mu_{s_j};0]$ and $\lambda_{s_i,s_j}=\mu_{s_i}\mu_{s_j}^{-1}$ for $i,j\in \{k,k+1,\dots, l+1\}$. For $\beta\in N_s^+\cup N_a$ we have $E_{-\beta}E_{\beta}\cdot v_\lambda=0$. We need to show that this choice of $\mu_{s_i}$ induces a choice $\mu_i', i\in \{1,2,\dots,n+1\}$ satisfying $z_{ij}=[\mu_i';1][\mu_j';0]$ and $\lambda_{ij}=\mu_i'\mu_j'^{-1}$ for $i\neq j$.

Let $\beta=\varepsilon_{t}-\varepsilon_{u}\in N_a,\beta'=\varepsilon_{u}-\varepsilon_{v}\in T_s^+$. We have $z_{ut}v^+=E_{\beta}E_{-\beta}v^+=[\lambda_{tu};0]v^+$. Since $-\beta,-\beta'\in T$ we can define $\kappa_{vut}$ similarly to Theorem \ref{th:torsionfree} by $E_{-\beta}E_{-\beta'}\cdot v^+=\kappa_{vut}E_{-\beta-\beta'}\cdot v^+$. 
By similar computations to those used to prove \eqref{eq:z1}--\eqref{eq:z3} we obtain the following for $1\leq i < j < k \leq n+1$:
\begin{align}
(q^{-1}-(q-q^{-1})\kappa_{kji})z_{kj}&=(\kappa_{kji}+1)(\lambda_{jk}^{-1}\kappa_{kji}+q^{-1}[\lambda_{jk};0])\label{eq:kappa9}\\
(q^{-1}-(q-q^{-1})\kappa_{kji})z_{ji}&=\kappa_{kji}(\lambda_{ij}\kappa_{kji}-[\lambda_{ij};-1])\label{eq:kappa10}\\
(q^{2}\lambda_{ij}+(q-q^{-1})\kappa_{kij})z_{ki}&=(\kappa_{kij}+q\lambda_{ij})(q^{-1}\lambda_{jk}\kappa_{kij}+q[\lambda_{ik};0])\\
(q^{2}\lambda_{ij}+(q-q^{-1})\kappa_{kij})z_{ij}&=\kappa_{kij}(q^{-1}\kappa_{kij}+[\lambda_{ij};1])\\
(q^{-2}\lambda_{jk}-(q-q^{-1})\kappa_{ikj})z_{ik}&=(\kappa_{ikj}+q^{-1}\lambda_{jk})(q\lambda_{ij}\kappa_{ikj}-q^{-1}[\lambda_{ik};0])\\
(q^{-2}\lambda_{jk}-(q-q^{-1})\kappa_{ikj})z_{kj}&=\kappa_{ikj}(q\kappa_{ikj}-[\lambda_{jk};-1])\\
(q\lambda_{jk}^{-1}+(q-q^{-1})\kappa_{jki})z_{jk}&=(\kappa_{jki}+\lambda_{jk}^{-1})(\kappa_{jki}-q[\lambda_{jk};0])\\
(q\lambda_{jk}^{-1}+(q-q^{-1})\kappa_{jki})z_{ki}&=\kappa_{jki}(\lambda_{ij}^{-1}\kappa_{jki}-[\lambda_{ik};-1])\\
(q^{-1}\lambda_{ij}^{-1}-(q-q^{-1})\kappa_{jik})z_{ji}&=(\kappa_{jik}+\lambda_{ij}^{-1})(\kappa_{jik}+q^{-1}[\lambda_{ij};0])\\
(q^{-1}\lambda_{ij}^{-1}-(q-q^{-1})\kappa_{jik})z_{ik}&=\kappa_{jik}(\lambda_{jk}^{-1}\kappa_{jik}+[\lambda_{ik};1])
\end{align}
The proof now depends on the order of $t,u,v$. We consider the case $t<u<v$, with similar arguments giving the other cases. Putting $i=t,j=u,k=v$, equations \eqref{eq:kappa9} and \eqref{eq:kappa10} hold. Since $z_{ji}=[\lambda_{ij};0]$ this gives:
\begin{equation*}
(q^{-1}-(q-q^{-1})\kappa_{kji})[\lambda_{ij};0]=\kappa_{kji}(\lambda_{ij}\kappa_{kji}-[\lambda_{ij};-1])
\end{equation*}
which may be solved to give $\kappa_{kji}=-1$ or $q^{-1}\lambda_{ij}^{-1}[\lambda_{ij};0]$. But, if $\kappa_{kji}=-1$ then \eqref{eq:kappa9} gives $z_{kj}=0$ contradicting $\varepsilon_j-\varepsilon_k\in T_s$. Hence $\kappa_{kji}=q^{-1}\lambda_{ij}^{-1}[\lambda_{ij};0]$. Inserting this in \eqref{eq:kappa9} gives:
\begin{align*}
q^{-1}\lambda_{ij}^{-2}z_{kj}&=\frac{-q^{-1}\lambda_{ij}^{-2}+q}{q-q^{-1}}\left ( \frac{-q^{-1}\lambda_{jk}^{-1}\lambda_{ij}^{-2}+q^{-1}\lambda_{jk}}{q-q^{-1}}\right )\\
z_{kj}&=\frac{q\lambda_{ij}-q^{-1}\lambda_{ij}^{-1}}{q-q^{-1}}\left ( \frac{\lambda_{ij}\lambda_{jk}-(\lambda_{ij}\lambda_{jk})^{-1}}{q-q^{-1}}\right ).
\end{align*}
In light of Theorem \ref{th:torsionfree}, for the root subsystem $T_s$ we have $z_{kj}=[\mu_k;1][\mu_j;0],\lambda_{jk}=\mu_j\mu_k^{-1}$ for some $\mu_k,\mu_j\neq 0$ in $\mathbb{F}$ (using that $\mathbb{F}$ is closed under quadratic extensions). Therefore $\lambda_{ij}=\pm \mu_k$ or $\pm q^{-1}\mu_j^{-1}$. Similarly, for any order of $t,u,v$ we have $\lambda_{tu}=\pm \mu_{v}$ or $\pm q^{-1}\mu_{u}^{-1}$. Defining $\beta=\varepsilon_{t}-\varepsilon_{u}\in T_s, \beta'=\varepsilon_{u}-\varepsilon_{v}\in N_a$ we similarly see $\lambda_{uv}=\pm \mu_{u}$ or $\pm q^{-1}\mu_{t}^{-1}$ for some $\mu_{t},\mu_{u}\neq 0$ such that $z_{tu}=[\mu_{t};1][\mu_{u};0],\lambda_{tu}=\mu_{t}\mu_{u}^{-1}$.

For $\varepsilon_t-\varepsilon_u$ and $\varepsilon_u-\varepsilon_v\in N_s^+\cup N_a$, by computations similar to \eqref{eq:lambda1}, $\lambda_{tu}\lambda_{uv}=\pm q^{-1},\pm\lambda_{uv}$ or $\pm \lambda_{tu}$. Since $\varepsilon_{s_{k}}-\varepsilon_{s_{k+1}}\in T_s^+$ and $\varepsilon_{s_{k-2}}-\varepsilon_{s_k}\in N_a$, we have $\lambda_{s_{k-2},s_{k-1}}\lambda_{s_{k-1},s_k}=\pm q^{-1}\mu_{s_{k}}^{-1}$ or $\pm\mu_{s_{k+1}}$. If $\lambda_{s_{k-2},s_{k-1}}\lambda_{s_{k-1},s_k}=\pm q^{-1}$, then $\mu_{s_k}=\pm 1$ or $\mu_{s_{k+1}}\pm q^{-1}$, each of which imply $z_{s_{k+1},s_{k}}=0$ contradicting $\varepsilon_{s_k}-\varepsilon_{s_{k+1}}\in T_s$. Therefore $\lambda_{s_{k-2},s_{k-1}}=\pm 1$. In the following, we let $I_1=\{1,2,\dots, k-1\},I_2=\{k,k+1,\dots, l+1\}, I_3=\{l+2,l+3,\dots,n+1\}$. By \eqref{eq:lambda2}, we have $\lambda_{s_i,s_j}=\pm 1$ for $i<j\in I_1$. Similarly, $\lambda_{s_i,s_j}=\pm 1$ for $i<j\in I_3$.

Now, let $i\in I_2,j\in I_3$. By construction $\varepsilon_{s_i}-\varepsilon_{s_{l+1}}\in T_s^+$ and $\varepsilon_{s_{l+1}}-\varepsilon_{s_{j}}\in N_a$, and we have $\lambda_{s_{l+1},s_{j}}=\pm \mu_{s_{l+1}}$ or $\pm q^{-1}\mu_{s_i}^{-1}$. If the first case holds for some $i\in I_2$ then we keep $\mu'_{s_i}=\mu_{s_i},i\in I_2$. Otherwise, it is the case that $\lambda_{s_{l+1},s_{j}}=\pm q^{-1}\mu_{s_r}^{-1}$ for all $r\in I$, which implies that $\mu_{s_r}=c$ for $r\in I$. Now, having $\varepsilon_{s_i}-\varepsilon_{s_{l}}\in T_s^+$ and $\varepsilon_{s_{l}}-\varepsilon_{s_{j}}\in N_a$ gives $\lambda_{s_l,s_{l+1}}\lambda_{s_{l+1},s_{j}}=\pm c$ or $\pm q^{-1}c^{-1}$, hence $\mu_{s_{l+1}}=\pm q^{-1}c^{-1}$ or $\pm c$. In the first case, we we keep $\mu'_{s_i}=\mu_{s_i},i\in I_2$. In the second case, we make a change of variables $\mu_{s_i}'=q^{-1}\mu_{s_{i}}^{-1},k\leq i \leq l+1$ giving $[\mu_{s_i};1][\mu_{s_{i+1}};0]=[\mu_{s_{i+1}}';1][\mu_{s_{i}}';0]$ and $\mu_{s_{i}}\mu_{s_{i+1}}^{-1}=\mu_{s_{i+1}}'\mu_{s_{i}}'^{-1}$, hence these  relations are preserved. In each case above, with our choice of $\mu'_{s_i}$, we have $\lambda_{s_{l+1},s_{j}}=\pm \mu'_{s_{l+1}}$. Therefore $\lambda_{s_i,s_j}=\lambda_{s_i,s_{l+1}}\lambda_{s_{l+1},s_j}=\pm \mu_i'$ for $i\in I_2$ and $j\in I_3$. Finally, for $i\in I_1$ and $j\in I_2$, since $\lambda_{s_{i},s_{j}}\neq \pm 1$ and $\lambda_{s_{j},s_{l+2}}\neq \pm 1$, we must have $\lambda_{s_{i},s_{j}}\lambda_{s_{j},s_{l+2}}=\pm q^{-1}$. Therefore $\lambda_{s_{i},s_{j}}=\pm q^{-1}\mu_{j}'^{-1}$.

Summarizing, we have:
\begin{equation*}
\lambda_{s_i,s_j}=\begin{cases}
\pm q^{-1}\mu_j'^{-1}& \text{for }i\in I_1,j\in I_2\\
\pm \mu_i'&\text{for }i\in I_2,j\in I_3\\
\pm \mu_i'\mu_j'^{-1}&\text{for }i<j\in I_2,\\
\pm q^{-1} &\text{for }i\in I_1,j\in I_3\\
\pm 1& \text{otherwise.}
\end{cases}
\end{equation*}
Therefore, we can put $\mu_i'=\pm q^{-1}$ for $i\in I_1$ and $\mu_j'=\pm 1$ for $j\in I_3$. Hence, for $i<j\in \{1,2,\dots, n+1\}$ we have $\lambda_{s_i,s_j}=\mu_{s_i}'\mu_{s_j}'^{-1}$ and $z_{s_i,s_j}=[\mu_{s_i}';1][\mu_{s_j}';0]$.

It follows that for $1\leq i  \leq n+1$ we have $F_iE_i\cdot v_\lambda=[\mu_i;1][\mu_{i+1};0]\cdot v_\lambda=[\bar{K}_i;1][\bar{K}_{i+1};0]\cdot v_\lambda$ for some weight vector $v_\lambda\in V$. It remains to show that the same holds for any weight vector in $V$. Let $\varepsilon_j-\varepsilon_k\in T\cup N_s^-$ be given, and suppose $j<k$ (the case $k<j$ is symmetric). If $i\neq j-1,j,k-1$ or $k$ then $E_{\varepsilon_j-\varepsilon_k}$ commutes with $F_iE_i$, $\bar{K}_i$ and $\bar{K}_{i+1}$, hence we have 
\begin{equation*}
F_iE_iE_{\varepsilon_j-\varepsilon_k}\cdot v^+=[\bar{K}_i;1][\bar{K}_{i+1};0]E_{\varepsilon_j-\varepsilon_k}\cdot v^+.
\end{equation*}
First, suppose $i=j-1$. We compute:
\begin{align*}
F_iE_iE_{\varepsilon_{i+1}-\varepsilon_k}\cdot v^+&=-F_iE_{\varepsilon_i-\varepsilon_k}\cdot v^+ +q^{-1}F_iE_{\varepsilon_{i+1}-\varepsilon_k}E_i\cdot v^+\\
	&=-\kappa_{i+1,i,k}E_{\varepsilon_{i+1}-\varepsilon_k}\cdot v^+ +q^{-1}E_{\varepsilon_{i+1}-\varepsilon_k}F_iE_i\cdot v^+\\
	&=\mu_{i+1}[\mu_i;1]E_{\varepsilon_{i+1}-\varepsilon_k}\cdot v^++q^{-1}[\mu_i;1][\mu_{i+1};0]E_{\varepsilon_{i+1}-\varepsilon_k}\cdot v^+\\
	&=[\mu_i;1][\mu_{i+1};1]E_{\varepsilon_{i+1}-\varepsilon_k}\cdot v^+\\
	&=[\bar{K}_i;1][\bar{K}_{i+1};0]E_{\varepsilon_{i+1}-\varepsilon_k}\cdot v^+,
\end{align*}
and similarly if $i=k$. If $i=k-1$ and $i>j$ then we compute:
\begin{align*}
F_iE_iE_{\varepsilon_j-\varepsilon_{i+1}}\cdot v^+&=qF_iE_{\varepsilon_j-\varepsilon_{i+1}}E_i\cdot v^+\\
	&=q^2K_{j,i+1}^{-1}E_{-\varepsilon_i+\varepsilon_j}E_i\cdot v^++qE_{\varepsilon_j-\varepsilon_{i+1}}F_iE_i\cdot v^+\\
	&=\mu_j^{-1}\mu_{i+1}\kappa_{j,i,i+1}E_{\varepsilon_j-\varepsilon_{i+1}}\cdot v^++q[\mu_i;1][\mu_{i+1};0]E_{\varepsilon_j-\varepsilon_{i+1}}\cdot v^+\\
	&=-\mu_{i+1}[\mu_i;1]E_{\varepsilon_j-\varepsilon_{i+1}}\cdot v^++q[\mu_i;1][\mu_{i+1};0]E_{\varepsilon_j-\varepsilon_{i+1}}\cdot v^+\\
	&=[\mu_i;1][\mu_{i+1};-1]E_{\varepsilon_j-\varepsilon_{i+1}}\cdot v^+\\
	&=[\bar{K}_i;1][\bar{K}_{i+1};0]E_{\varepsilon_j-\varepsilon_{i+1}}\cdot v^+,
\end{align*}
and similarly if $i=j$ and $i<k$. If $i=j$ and $i=k-1$ it is easy to see that $F_iE_iE_i\cdot v^+=[\bar{K}_i;1][\bar{K}_{i+1};0]E_i\cdot v^+$. Since $V$ is irreducible, it is generated by $v^+$, hence it is equal to $U_q(\mathfrak{p})\cdot v^+$. The result follows by induction on the degree of monomials in $U_q(\mathfrak{g})$.
\end{proof}

\section{Construction of irreducible completely pointed modules}

In this section we find a quantum version of the  construction in \cite{BBL} of irreducible completely pointed weight $\mathfrak{gl}_{n+1}$-modules.
Then we show that any irreducible completely pointed $U_q(\mathfrak{gl}_{n+1})$-module occurs in this way.

As in \cite[Theorem 3.2]{Hayashi}, one checks that there is an $\mathbb{F}$-algebra homomorphism
\begin{equation}\label{eq:pi_map}
\begin{aligned}
\pi:U_q(\mathfrak{gl}_{n+1})&\longrightarrow A_{n+1}^q \\
E_i &\longmapsto x_iy_{i+1}, \\
F_i &\longmapsto x_{i+1}y_i, \\
\bar K_i &\longmapsto \omega_i.
\end{aligned}
\end{equation}
where $A_{n+1}^q$ is the \emph{quantized Weyl algebra}, defined as the associative unital $\mathbb{F}$-algebra with generators $\omega_i, \omega_i^{-1}, x_i, y_i$, $i\in\{1,2,\ldots,n+1\}$ and defining relations
\begin{subequations}\label{eq:Anq_rels}
\begin{align}
\omega_i\omega_j &= \omega_j\omega_i, \\
\omega_i\omega_i^{-1}&=\omega_i^{-1}\omega_i = 1, \\ 
\omega_ix_j\omega_i^{-1} &= q^{\delta_{ij}}x_j, \\ 
\omega_iy_j\omega_i^{-1} &= q^{-\delta_{ij}}y_j,\\ 
y_i x_j &= x_j y_i, \quad i\neq j, \\
y_ix_i - q^{-1} x_i y_i&=\omega_i, \label{eq:Anq_rels_yixi1}\\
y_ix_i - q x_iy_i &= \omega_i^{-1}
\end{align}
where $i,j\in\{1,2,\ldots,n+1\}$.
The last two relations are equivalent to the two relations
\begin{equation}\label{eq:Anq_rels_ti}
y_ix_i =
 \frac{q\omega_i - (q\omega_i)^{-1}}{q-q^{-1}},\qquad
x_iy_i=
 \frac{\omega_i - \omega_i^{-1}}{q-q^{-1}}
\end{equation}
\end{subequations}
Thus, $A_{n+1}^q$ is isomorphic to the rank $n$ generalized Weyl algebra $R(\sigma,t)$ where $R=\mathbb{F}[\omega_1^{\pm 1},\ldots,\omega_{n+1}^{\pm 1}]$, $\sigma_j(\omega_i)=q^{-\delta_{ij}}\omega_i$, $t_i= \frac{q\omega_i - (q\omega_i)^{-1}}{q-q^{-1}}$ for $i\in\{1,2,\ldots,n+1\}$.

The central element $I_{n+1}=\bar K_1\cdots \bar K_{n+1}$ of $U_q(\mathfrak{gl}_{n+1})$ is mapped by $\pi$ to the element $\mathbb{E}_q:=\omega_1\omega_2\cdots\omega_{n+1}$. $\mathbb{E}_q$ should be thought of as $q^{\sum_i x_i\partial i}$: a $q$-analogue of the Euler operator.

\begin{lemma}\label{lem:qEuler}
The following identities hold.
\begin{gather}
\label{eq:qEuler1}
\mathbb{E}_q x_i \mathbb{E}_q^{-1} = q x_i,
\quad
\mathbb{E}_q y_i \mathbb{E}_q^{-1} = q^{-1} y_i,
\quad 
\mathbb{E}_q \omega_i \mathbb{E}_q^{-1} = \omega_i,
\qquad i\in\{1,2,\ldots,n+1\},\\ 
\label{eq:qEuler2}
A_{n+1}^q = \bigoplus_{m\in\mathbb{Z}} A_{n+1}^q[m],\qquad
A_{n+1}^q[m] = \big\{a\in A_{n+1}^q\mid \mathbb{E}_q a\mathbb{E}_q^{-1} = q^m a\big\}, \\
\label{eq:qEuler3}
A_{n+1}^q[m_1]\cdot A_{n+1}^q[m_2]\subseteq A_{n+1}^q[m_1+m_2], \\ 
\label{eq:qEuler4}
\pi\big(U_q(\mathfrak{gl}_{n+1})\big) = A_{n+1}^q[0] = C_{A_{n+1}^q}(\mathbb{E}_q).
\end{gather}
where $C_{A_{n+1}^q}(\mathbb{E}_q)$ denotes the centralizer of $\mathbb{E}_q$ in $A_{n+1}^q$.
\end{lemma}
\begin{proof}
The identities \eqref{eq:qEuler1} follow directly from the commutation relations \eqref{eq:Anq_rels} in $A_{n+1}^q$. Identities \eqref{eq:qEuler2} and \eqref{eq:qEuler3} follow from \eqref{eq:qEuler1} and that $A_{n+1}^q$ is generated by $x_i, y_i$ and $\omega_i$. The second equality of \eqref{eq:qEuler4} is trivial. By definition, \eqref{eq:pi_map}, of $\pi$ it is clear that $\pi(E_i),\pi(F_i),\pi(\bar K_i)\in A_{n+1}^q[0]$ for all $i\in\{1,\ldots,n\}$. Since $U_q(\mathfrak{gl}_{n+1})$ is generated by the set $\{E_i,F_i,\bar K_i\}_{i=1}^{n+1}$, it follows that $\pi\big(U_q(\mathfrak{gl}_n)\big)\subseteq A_n^q[0]$. It remains to prove that 
$A_{n+1}^q[0]\subseteq \pi\big(U_q(\mathfrak{gl}_{n+1})\big)$.
First observe that $A_{n+1}^q[0]$ is invariant under left multiplication by elements from $R=\mathbb{F}[\omega_i^{\pm 1}\mid i=1,\ldots,n+1]$.
Since $A_{n+1}^q$ is a generalized Weyl algebra, it follows that $A_{n+1}^q[0]$ is generated as a left $R$-module by all monomials
\[a=x_1^{k_1}x_2^{k_2}\cdots x_{n+1}^{k_{n+1}}y_1^{l_1}y_2^{l_2}\cdots y_{n+1}^{l_{n+1}}\]
where $k,l\in (\mathbb{Z}_{\ge 0})^{n+1}$ are such that $\sum_i k_i=\sum_i l_i$ and $k_i\cdot l_i=0$ for all $i\in\{1,2,\ldots,n+1\}$.
Since any such monomial $a$ is a product of elements of the form $x_iy_j$, where $i\neq j$, it suffices to show that $x_iy_j$ lies in the image of $\pi$ for any $i\neq j$. We prove by induction on $j$ that $x_iy_j\in\pi\big(U_q(\mathfrak{gl}_{n+1})\big)$ whenever $i<j$. If $j=i+1$, then $x_iy_{i+1}=\pi(E_i)$. If $j>i+1$, note that by \eqref{eq:Anq_rels_yixi1},
\begin{equation}
x_iy_j=\omega_{j-1}^{-1}[x_iy_{j-1}, \pi(E_{j-1})]_q
\end{equation}
(recalling that $[a,b]_u:=ab-uba$), which by the induction hypothesis lies in the image of $\pi$. Similarly one can use $\pi(F_i)$ to prove that $x_iy_j\in\pi\big(U_q(\mathfrak{gl}_{n+1})\big)$ if $i>j$. This finishes the proof of \eqref{eq:qEuler4}.
\end{proof}

\begin{lemma}\label{lem:Anq_cptd}
Let $V$ be an irreducible $A_{n+1}^q$ weight module and $\mathfrak{m}\in\mathrm{Specm}(R)$ with $V_\mathfrak{m}\neq 0$. Then $\dim_{R/\mathfrak{m}} V_\mathfrak{m}=1$. If in addition $\dim_{\mathbb{F}}R/\mathfrak{m}=1$, then $V$ is completely pointed.
\end{lemma}
\begin{proof} Let $A=A_{n+1}^q$. Since $A$ is a generalized Weyl algebra and $V$ is an irreducible weight $A$-module, each weight space $V_{\mathfrak{m}}$ is an irreducible $C(\mathfrak{m})$-module, where $C(\mathfrak{m})=\bigoplus_{g\in\mathrm{Stab}_{\mathbb{Z}^{n+1}}(\mathfrak{m})} A_g$ is the cyclic subalgebra of $A$ with respect to $\mathfrak{m}$, (see e.g. \cite[Prop.~7.1]{MazPonTur2003}).
Since $q$ is not a root of unity, the action of $\mathbb{Z}^{n+1}$ on $\mathrm{Specm}(R)$ is faithful, and therefore $C(\mathfrak{m})=R$, which is commutative. This implies that $\dim_{R/\mathfrak{m}} V_{\mathfrak{m}}\le 1$. The second claim follows the fact that the support of an indecomposable weight module over a generalized Weyl algebra is invariant under the automorphisms $\sigma_1,\ldots,\sigma_n$ and that $\dim_\mathbb{F} R/\mathfrak{m}=\dim_{\mathbb{F}} R/\tau(\mathfrak{m})$ for any $\mathbb{F}$-algebra automorphism $\tau$ of $R$.
\end{proof}

Let $\mathrm{Specm}^1(R)$ denote the set of all maximal ideals $\mathrm{m}$ of $R$ such that $R/\mathfrak{m}$ is one-dimensional over 
$\mathbb{F}$. Thus $\mathfrak{m}=(\omega_1-\mu_1,\ldots,\omega_{n+1}-\mu_{n+1})$, where $(\mu_1,\ldots,\mu_{n+1})\in(\mathbb{F}^\times)^{n+1}$. 
%

\begin{mainthmiia}\label{thm:pointed}
Let $W$ be an irreducible completely pointed $A_n^q$-module. Let $\pi^\ast W$ be the $U_q(\mathfrak{gl}_{n+1})$-module, given as the $\pi$-pullback of $W$, where $\pi$ is the map \eqref{eq:pi_map}.
Then $\pi^\ast W$ is completely reducible, and each irreducible submodule is completely pointed, and occurs with multiplicity one.
\end{mainthmiia}
\begin{proof}
Since $\mathbb{E}_q\in R$, the $A_{n+1}^q$-module $W$ decomposes in particular into eigenspaces with respect to $\mathbb{E}_q$. Due to the commutation relations in $A_{n+1}^q$, the ratio of any two eigenvalues is a power of $q$. That is, there exists a non-zero $\xi\in\mathbb{F}$ such that
\begin{equation}
W=\bigoplus_{m\in\mathbb{Z}}W[m],\qquad
W[m]=\big\{w\in W\mid \mathbb{E}_q w = \xi q^m w\big\}
\end{equation}
Each $W[m]$ is a direct sum of certain $R$-weight spaces of $W$. More precisely, for each $m\in\mathbb{F}$ we have
\[
W[m]=\bigoplus_{\substack{\mathfrak{m}\in\mathrm{Specm}^1(R)\\ \mathbb{E}_q-q^m\xi\in\mathfrak{m}}} W_\mathfrak{m}.
\] 
By Lemma \ref{lem:qEuler}, each subspace $W[m]$ is an $U_q(\mathfrak{gl}_{n+1})$-submodule of $W$.
Since $\pi(\bar{K}_i)=\omega_i$ for each $i\in\{1,\ldots,n+1\}$,
and $W$ is completely pointed as an $A_{n+1}^q$-module, it follows that each $W[m]$, $m\in\mathbb{Z}$ is a completely pointed $U_q(\mathfrak{gl}_{n+1})$-module.
It remains to prove that for each $m\in\mathbb{Z}$, the $U_q(\mathfrak{gl}_{n+1})$-module $W[m]$ is either zero or irreducible.
By \eqref{eq:qEuler4}, proving that $W[m]$ is irreducible as an $U_q(\mathfrak{gl}_{n+1})$-module is the same thing as proving that $W[m]$ is irreducible as an $A_{n+1}^q[0]$-submodule of $W$.

Suppose $W[m]\neq \{0\}$ and let $w_0$ and $w_1$ be any two non-zero weight vectors of $W[m]$ of weights $\mathfrak{m}_0$ and $\mathfrak{m}_1$ respectively. $A_{n+1}$ is generated as a left $R$-module by monomials of the form
\[a=x_1^{k_1}x_2^{k_2}\cdots x_{n+1}^{k_{n+1}}\cdot y_1^{l_1}y_2^{l_2}\cdots y_{n+1}^{l_{n+1}},\]
where $k,l\in(\mathbb{Z}_{\ge 0})^{n+1}$ and $k_il_i=0$ for each $i$.
Moreover, there is at most one such monomial $a$ such that $(aW_{\mathfrak{m}_0})\cap W_{\mathfrak{m}_1}\neq\{0\}$. Since $W$ is irreducible as an $A_{n+1}$-module, there exists $r\in R$ and a single monomial $a$ such that $raw_0=w_1$.
Since $w_0,w_1\in W[m]$, this forces $\sum_i k_i=\sum_i l_i$, which implies that $a\in A_{n+1}^q[0]$. This proves that $W[m]$ is irreducible as an $A_{n+1}^q[0]$-module.
\end{proof}

The cyclic algebra of $U_q(\mathfrak{g})$---$C(U_q(\mathfrak{g}))$---is defined to be the subalgebra of all elements commuting with $K_i^{\pm 1}, i\in \{1,2,\dots, n+1\}$.

\begin{lemma}\label{lem:kerpi_AnnV}
Let $V$ be an irreducible, infinite dimensional, completely pointed $U_q(\mathfrak{sl}_{n+1})$-module. Then $\ker \pi\subseteq \mathrm{Ann}_{U_q(\mathfrak{sl}_{n+1})} V$, where $\pi$ is the map \eqref{eq:pi_map}.
\end{lemma}

\begin{proof}
Let $x\in C(U_q(\mathfrak{sl}_{n+1}))$.   Write $x=\sum KF_{\mathbf{i}}E_\mathbf{j}$ where $K\in\mathbb{F}[{K}_1^{\pm1},{K}_2^{\pm1},\dots, {K}_n^{\pm1}]$, and the sequences $\mathbf{i}=(i_1,i_2,\dots, i_l),\mathbf{j}=(j_1,j_2,\dots, j_l)$ are such that $\mathbf{i}$ is a permutation of $\mathbf{j}$.  We have:
\begin{align*}
\pi(x)&=\sum \pi(K)\pi(F_{\mathbf{i}})\pi(E_\mathbf{j})\\
	&=\sum \pi(K)\prod_{r=1}^l(x_{i_r+1}y_{i_r})\prod_{r=1}^l (x_{j_r}y_{j_r+1})\\
	&=\sum \pi(K)\prod_{r=1}^l\big([\omega_{i_r};s_{r}-s'_r][\omega_{j_r+1};t_{r}-t'_{r}]\big)
\end{align*}
where $s_{r}$ (resp. $s'_r$) denotes the number of times the element $i_r$ (resp. $i_r-1$) appears in the sequence $(j_1,j_2,\dots,j_l) \backslash(i_{r+1},i_{r+2},\dots, i_l)$ and $t_{r}$ (resp. $t'_r$) denotes the number of times $j_r+1$ (resp. $j_{r}$) appears in the sequence $ (j_{r+1},j_{r+2},\dots, j_l)$.  We prove this by induction on $l$.
If $l=1$ then we have $s_1=1,s'_1=0,t_1=0,t'_1=0$ and compute:
\begin{align*}
x_{i_1+1}y_{i_1}x_{i_1}y_{i_1+1}&=
x_{i_1+1}\frac{(q\omega_{i_1})-(q\omega_{i_1})^{-1}}{q-q^{-1}}y_{i_1+1}\\
&=[\omega_{i_1};1][\omega_{i_1+1};0]
\end{align*}
For $l>1$ observe that $x_{i_l+1}y_{i_l}$ commutes with $x_{j_k}y_{j_k+1}$ if and only if $j_k\neq i_l$.  Let $k\in \{1,2,\dots l\}$ be the minimum such that $i_l=j_k$ (we know such a $k$ exists since $\mathbf{j}$ is a permutation of $\mathbf{i}$).  We have:
\begin{align*}
\lefteqn{\left (\prod_{r=1}^lx_{i_r+1}y_{i_r}\right )\left (\prod_{r=1}^l x_{j_r}y_{j_r+1}\right )}\\
&=\left (\prod_{r=1}^{l-1}x_{i_r+1}y_{i_r}\right )\left (\prod_{r=1}^{k-1} x_{j_r}y_{j_r+1}\right )(x_{i_l+1}y_{i_l})(x_{j_k}y_{j_{k}+1})\left (\prod_{r=k+1}^lx_{j_r}y_{j_{r}+1}\right )\\
&=\left (\prod_{r=1}^{l-1}x_{i_r+1}y_{i_r}\right )\left (\prod_{r=1}^{k-1}x_{j_r}y_{j_r+1}\right ) t_{i_l}\sigma_{j_k+1}(t_{j_k+1})\left (\prod_{r=k+1}^lx_{j_r}y_{j_{r}+1}\right )\\
&=\left (\prod_{r=1}^{l-1}x_{i_r+1}y_{i_r}\right )\left (\prod_{k\neq r=1}^{l}x_{j_r}y_{j_r+1}\right )
[\omega_{i_l};{s_l-s_l'}][\omega_{j_k+1};t_k-t_k']
\end{align*}
and apply induction to obtain the desired result. 

Now, let $V$ be an infinite-dimensional irreducible completely-pointed $U_q(\mathfrak{sl}_{n+1})$-module. Then, using Theorem I, we extend the $U_q(\mathfrak{sl}_{n+1})$ action on $V$ to a $U_q(\mathfrak{gl}_{n+1})$ action that satisfies $F_iE_i\cdot v=[\bar{K}_i;1][\bar{K}_{i+1};0]\cdot v$ for all weight vectors $v\in V$. By a similar computation we see:
\begin{equation*}
x\cdot v=\sum K\prod_{r=1}^l\big([\bar K_{i_r};s_{r}-s'_r][\bar K_{j_r+1};t_{r}-t'_{r}]\big)\cdot v
\end{equation*}
where $s_r,s_r',t_r,$ and $t_r'$ are as before. Therefore, $\pi(x)$ is a Laurent polynomial in the $\omega_i$ and $x$ acts on $v$ by the same Laurent polynomial evaluated at $\omega_i=\mu_i$.  If $x\in \text{ker}(\pi)$ then this Laurent polynomial must be identically 0, hence $x\cdot v=0$.  Since $V$ is irreducible, $v$ is a cyclic vector, hence $x\in \text{Ann}_{U_q(\mathfrak{sl}_{n+1})}(V)$.

Next we prove $\ker(\pi)\subseteq\mathrm{Ann}_{U_q(\mathfrak{sl}_{n+1})}(V)$. Let $x\in\ker(\pi)$. Without loss of generality we can assume $x$ is homogeneous with respect to the root lattice grading: $x=x_\beta$ for some $\beta\in Q$ and $K_i x K_i^{-1}=q^{\langle \beta, \alpha_i\rangle} x$ for $i=1,2,\ldots,n$. Assume for the sake of contradiction that there exists an irreducible completely pointed weight module $V$ for which $x\cdot V\neq 0$. Then there exists a weight vector $v\in V$ such that $w=x \cdot v\neq 0$. Since $x$ is homogeneous, $w$ is also a weight vector. Since $V$ is irreducible, there exists some homogeneous element $y\in U_q(\mathfrak{sl}_{n+1})$ of degree $-\beta$ such that $y\cdot w=v$. Then $yx$ has degree zero, and thus belongs to the centralizer $C(U_q(\mathfrak{sl}_{n+1}))$ of $K_1,\ldots, K_{n}$. Also, $yx$ belongs to $\ker(\pi)$ since it is an ideal in $U_q$. So $yx\in C(U_q(\mathfrak{sl}_{n+1}))\cap \ker(\pi)$ which by the previous paragraph implies that $(yx)\cdot V=\{0\}$, which contradicts the fact that $(yx)\cdot v_\lambda=v_\lambda\neq 0$.
\end{proof}

\begin{mainthmiib}
Any infinite-dimensional irreducible completely pointed $U_q(\mathfrak{sl}_{n+1})$ is isomorphic to a direct summand of $\pi^{\ast} W$ for some irreducible completely pointed $A_{n+1}^q$-module $W$.
\end{mainthmiib}
\begin{proof}
Let $V$ be an irreducible completely pointed $U_q(\mathfrak{sl}_{n+1})$-module, where we extend the action to make $V$ become a $U_q=U_q(\mathfrak{gl}_{n+1})$-module as in Theorem I. Consider the $A_{n+1}^q$-module 
\[\widetilde{W}=A_{n+1}^q\otimes_{U_q} V\]
where $A_{n+1}^q$ is regarded as a right $U_q$-module via the homomorphism $\pi:U_q\to A_{n+1}^q$ defined in \eqref{eq:pi_map}. 
Since $A_{n+1}^q=\bigoplus_{m\in\mathbb{Z}} A_{n+1}^q[m]$ as right $A_{n+1}^q[0]$-modules, we have
\begin{equation}
\widetilde{W}=\bigoplus_{m\in\mathbb{Z}}\widetilde{W}[m],\qquad \widetilde{W}[m]= A_n^q[m]\otimes_{U_q} V.
\end{equation}
Since $V$ is completely pointed we have $V=\bigoplus_{\lambda\in(U_q^0)^\ast} V_\lambda$, $V_\lambda=\{v\in V\mid kv=\lambda(k)v,\forall k\in U_q^0\}$, where $(U_q^0)^\ast$ is the set of characters of $U_q^0=\mathbb{F}[K_1^{\pm 1},\ldots, K_{n+1}^{\pm 1}]$. Since $I_{n+1}:=\bar K_1\cdots \bar K_{n+1}$ is central in $U_q$ and $V$ is irreducible, it follows that $I_{n+1}$ acts by some scalar $\xi\in\mathbb{F}\setminus\{0\}$ on $V$. Thus $\mathbb{E}_q a_m\otimes_{U_q} v_\lambda = q^m a_m\mathbb{E}_q\otimes_{U_q} v_\lambda = q^ma_m\otimes_{U_q} I_nv_\lambda=\xi q^m a_m\otimes_{U_q} v_\lambda$ for $m\in\mathbb{Z}$, $a_m\in A_{n+1}^q[m]$, $\lambda\in (U_q^0)^\ast$ and $v_\lambda\in V_\lambda$.
Thus
\begin{equation}
A_{n+1}^q[m]\otimes V=\big\{w\in\widetilde W\mid \mathbb{E}_q w= \xi q^m w\big\}.
\end{equation}
This turns $\widetilde{W}$ into a $\mathbb{Z}$-graded $A_{n+1}^q$-module: $A_{n+1}^q[m_1]\widetilde{W}[m_2]\subseteq \widetilde{W}[m_1+m_2]$ for all $m_1,m_2\in\mathbb{Z}$. In particular $\widetilde{W}[0]$ is left $A_{n+1}^q[0]$-submodule of $\widetilde{W}$, and can thus be regarded as a $U_q$-module via the map \eqref{eq:pi_map}. By \eqref{eq:qEuler4} and Lemma \ref{lem:kerpi_AnnV}, there is a linear map
\begin{equation}\label{eq:VisoW0}
\begin{aligned}
A_{n+1}^q[0]\otimes_{\mathbb{F}} V &\longrightarrow V, \\ 
\pi(a)\otimes v &\longmapsto av.
\end{aligned}
\end{equation}
It is balanced with respect to the right and left $U_q$-actions and hence induces a homomorphism $\widetilde{W}[0]\to V$, which is an isomorphism of $U_q$-modules with inverse $v\mapsto 1\otimes_{U_q} v$.
Let $N$ be the sum of all $\mathbb{Z}$-graded submodules $S$ of $\widetilde{W}$ such that $S\cap \widetilde{W}[0]=\{0\}$. Then $N$ is the unique maximal $\mathbb{Z}$-graded submodule of $\widetilde{W}$ with $N\cap \widetilde{W}[0]=\{0\}$. Define
\begin{equation}
W=\widetilde{W}/N.
\end{equation}
Since $N$ is graded, so is $W$, and \eqref{eq:VisoW0} and $N\cap\widetilde{W}[0]=\{0\}$ imply that $V$ is isomorphic to $W[0]$, which is a direct summand of $W$ as a left $A_{n+1}^q[0]$-module, hence as a left $U_q$-module via \eqref{eq:pi_map}.

We show that $W$ is an irreducible $A_{n+1}^q$-module. If $M$ is a nonzero submodule of $W$, then the inverse image, $\widetilde{M}$, of $M$ under the canonical projection $\widetilde{W}\to W$ is a graded nonzero module containing $N$ and thus must have nonzero intersection with $\widetilde{W}[0]$ by maximality of $N$. Since $V$ is an irreducible $U_q$-module, $\widetilde{W}[0]$ is an irreducible $A_n^q[0]$-module. Therefore
\[\widetilde{M}\supseteq A_{n+1}^q[0]\widetilde{M}\supseteq A_{n+1}^q[0]\widetilde{W}[0]=\widetilde{W}[0].\]
But $\widetilde{W}$ is generated as a left $A_{n+1}^q$-module by $\widetilde{W}[0]$ which implies that $\widetilde{M}=\widetilde{W}$ and thus $M=W$. 

It remains to be proved that $W$ is completely pointed.
Let $\lambda$ be a character of $V$ such that $V_\lambda\neq\{0\}$ and let $v_\lambda\in V_\lambda\setminus\{0\}$. Define $\mathfrak{m}=(\omega_1-\lambda(\bar K_1),\ldots,\omega_{n+1}-\lambda(\bar K_{n+1}))\in\mathrm{Specm}^1(R)$. Then the vector $(1\otimes_{U_q} v_\lambda)+N$ is a nonzero $R$-weight vector of weight $\mathfrak{m}$, so by Lemma \ref{lem:Anq_cptd}, $W$ is a completely pointed $A_{n+1}^q$-module.
\end{proof}

\section{Acknowledgment}
	The first author is supported in part by the CNPq grant (301320/2013-6) and by the Fapesp grant (2010/50347-9). The second author would like to thank Dennis Hasselstrom Pedersen for interesting discussions. The third author is grateful to the University of S\~{a}o Paulo for hospitality, his supervisor Vyacheslav Futorny, and to the Fapesp for financial support (grant number 2011/12079-5).

\end{document}